\documentclass[11pt]{amsart}

\usepackage{graphicx}
\usepackage[english]{babel}
\usepackage[T1]{fontenc}
\usepackage[latin1]{inputenc}
\usepackage{amsfonts}
\usepackage{amssymb}
\usepackage{amsthm}
\usepackage{amsmath}
\usepackage{tikz}
\usepackage{float}
\usepackage{enumerate}
\usepackage{accents}
\usepackage{mathtools}
\usepackage{mathrsfs}
\usepackage{comment}
\usepackage{afterpage}
\usepackage{bbm}
\usepackage{dsfont}
\usepackage{stmaryrd}
\usepackage{hyperref}
\usepackage{mleftright}
\usepackage{subfig}
\usepackage{soul}
\usepackage{tikz-cd} 
\usepackage[foot]{amsaddr}
\usepackage{fullpage}

\theoremstyle{plain}
\newtheorem{thm}{Theorem}[section]
\newtheorem{lem}[thm]{Lemma}
\newtheorem{prop}[thm]{Proposition}
\newtheorem{cor}[thm]{Corollary}
\newtheorem*{thm*}{Theorem}
\newtheorem*{prop*}{Proposition}
\newtheorem*{cor*}{Corollary}
\newtheorem{thmintro}{Theorem}

\newtheorem{corintro}[thmintro]{Corollary}

\theoremstyle{definition}
\newtheorem{defn}[thm]{Definition}
\newtheorem{ex}[thm]{Example}
\newtheorem{rmk}[thm]{Remark}
\newtheorem*{rmk*}{Remark}

\newtheorem*{quest*}{Question}

\newtheorem*{defn*}{Definition}

\newcommand{\acts}{\curvearrowright}
\newcommand{\ra}{\rightarrow}
\newcommand{\Ra}{\Rightarrow}

\newcommand{\cu}{\subseteq}

\newcommand{\wh}{\widehat}
\newcommand{\x}{\times}
\renewcommand{\o}{\circ}
\newcommand{\id}{\mathrm{id}}

\newcommand{\mc}{\mathcal}
\newcommand{\mf}{\mathfrak}
\newcommand{\mscr}{\mathscr}

\newcommand{\R}{\mathbb{R}}
\newcommand{\Z}{\mathbb{Z}}
\newcommand{\N}{\mathbb{N}}
\renewcommand{\H}{\mathbb{H}}
\newcommand{\s}{\sigma}

\newcommand{\Om}{\Omega}

\newcommand{\CAT}{{\rm CAT(0)}}

\DeclareMathOperator{\isom}{Isom}

\DeclareMathOperator{\aut}{Aut}

\DeclareMathOperator{\rk}{rk}

\newcommand{\Min}{\mathrm{Min}}
\newcommand{\res}{\mathrm{res}}
\newcommand{\I}{\mscr{I}}

\allowdisplaybreaks

\begin{document}

\title{Convex cores for actions on finite-rank median algebras} 
\author[E. Fioravanti]{Elia Fioravanti}\address{Max Planck Institute for Mathematics, Bonn, Germany}\email{fioravanti@mpim-bonn.mpg.de} 

\begin{abstract}
We show that every action of a finitely generated group on a finite-rank median algebra admits a nonempty ``convex core'', even when no metric or topology is given. We then use this to deduce an analogue of the flat torus theorem for actions on connected finite-rank median spaces. We also prove that isometries of connected finite-rank median spaces are either elliptic or loxodromic. 
\end{abstract}

\maketitle

\section{Introduction.}

A metric space $X$ is \emph{median} if, for all $x_1,x_2,x_3\in X$, there exists a unique $m=m(x_1,x_2,x_3)\in X$ with the property that
\begin{equation}\label{median space defn}\tag{$\ast$} d(x_i,x_j)=d(x_i,m)+d(m,x_j)\end{equation}
for all $1\leq i<j\leq 3$. We refer to median metric spaces simply as \emph{median spaces}. The \emph{rank} of a connected median space is the supremum of the topological dimensions of its compact subsets. 

The simplest examples of finite-rank median spaces are provided by $\R$--trees and finite-dimensional $\CAT$ cube complexes with the $\ell^1$ metric \cite{Chepoi}. The Guirardel core of a pair of actions on $\R$--trees is also a median space of rank $\leq 2$ \cite{Guirardel-core}. On the other hand, for every non-discrete measure space $(\Om,\mu)$, the Banach space $L^1(\Om,\mu)$ is an infinite-rank median space \cite{CDH}.

Most importantly, finite-rank median spaces arise as ultralimits of sequences of $\CAT$ cube complexes of uniformly bounded dimension, or, in fact, any sequence of spaces that can be suitably approximated by such cube complexes. Thus, the asymptotic cones of any \emph{coarse median group} can be given a natural structure of finite-rank median space \cite{Bow-cm,Zeidler}. Such groups include all cocompactly cubulated groups and mapping class groups of compact surfaces \cite{Bow-cm,HHS1}. 

There is a sharp divide between the properties of median spaces of \emph{finite rank}, which closely track those of $\CAT$ cube complexes, and the properties of \emph{infinite-rank} median spaces, such as $L^1([0,1])$. For this reason, we restrict our attention to finite-rank spaces in this work. 

\medskip
Our first goal is to provide a proof of the following two facts (Corollaries~\ref{main median semisimple} and~\ref{flattorus cor}), which follow from our more general results on group actions on median algebras. 

To begin with, isometries of finite-rank median spaces are \emph{semisimple}:

\begin{corintro}\label{main median semisimple}
Let $X$ be a connected, finite-rank median space. For every $g\in \isom X$:
\begin{enumerate}
\item the translation length of $g$ is realised by some point of $X$;
\item if $X$ is geodesic and $g$ does not fix a point, then $g$ admits an \emph{axis}: a bi-infinite $\langle g\rangle$--invariant geodesic along which $g$ translates nontrivially. 
\end{enumerate}
\end{corintro}

When $X$ is a $\CAT$ cube complex, Corollary~\ref{main median semisimple} is due to Haglund \cite{Haglund} and also holds when $X$ is infinite-dimensional\footnote{Haglund's theorem is normally stated with the additional assumption that $g$ act \emph{stably without inversions}. This is only necessary if we want an axis to intersect the $0$--skeleton of the cube complex. In our setting, the fact that the median space is connected guarantees that there cannot be any inversions (see Remark~\ref{connected inversions}).}. In median spaces, however, the finite-rank assumption is essential for Corollary~\ref{main median semisimple} to hold (see for instance Example~\ref{L^1 not semisimple ex} below). 

Haglund's proof is combinatorial in nature and does not carry over to the general median setting. We will develop an alternative approach, based on the study of ``\emph{convex cores}'' for general actions $G\acts X$. These are certain $G$--invariant convex subsets of $X$ naturally attached to the $G$--action.

Our approach also easily yields the following version of the flat torus theorem. Similar results when $X$ is a $\CAT$ cube complex were obtained in \cite{WW,Woodhouse-flattorus,Genevois-flattorus}.
 
\begin{corintro}\label{flattorus cor}
Let $A$ be a finitely generated, virtually abelian group. Let $A\acts X$ be an isometric action on a connected, finite-rank median space. Then $A$ stabilises a nonempty convex subset $C\cu X$ isometric to a subset of $(\R^n,d_{\ell^1})$ with $n\leq\rk X$.
\end{corintro}

Some version of Corollaries~\ref{main median semisimple} and~\ref{flattorus cor} was certainly to be expected if one subscribes to the sensible view that ``anything that holds in $\CAT$ cube complexes and $\R$--trees should also hold in finite-rank median spaces''. What I found surprising is that these results do not require the metric on $X$ to be \emph{complete}, which might appear necessary since both involve finding a fixed point of sorts.

\medskip
In fact, most of the results in this paper hold for general automorphisms of finite-rank median \emph{algebras}, even when no metric or topology is present. 

As an example of the breadth of this extension, consider the case when $X=\R$. While, up to index $2$, every isometry of $\R$ is a translation, the automorphism group of the underlying median algebra is the entire homeomorphism group of $\R$.

Another interesting example is provided by the Baumslag--Solitar groups $BS(1,n)$ with $n\geq 2$. By \cite[Theorem~A]{Fio2}, these groups do not admit any free (or proper) actions by isometries on finite-rank median spaces. However, $BS(1,n)$ does admit a properly discontinuous, cocompact action by median-preserving homeomorphisms on the rank--$2$ median algebra $T_n\x\R$, where $T_n$ is the standard Bass--Serre tree (Example~\ref{BS example}).

This level of generality will prove particularly useful in our upcoming work \cite{Fio10a}, where we generalise to all fundamental groups of compact special cube complexes the classical fact that every outer automorphism of a hyperbolic group $G$ can be realised as a homothety of a small $G$--tree \cite{Paulin-ENS,BH-Ann,GJLL,GLL}. Actions on $\R$--trees will have to be replaced by actions on higher-dimensional median spaces. It will then be important that many results can be expressed purely in median-algebra terms, leaving us free to modify the median metric to suit our needs.

\medskip
In the rest of the introduction, we state our results in the general context of median algebras. This inevitably requires giving a few important definitions first.

\medskip
{\bf Statement of results on finite-rank median algebras.}
A \emph{median algebra} is a pair $(M,m)$, where $M$ is a set and $m\colon M^3\ra M$ is a map satisfying:
\begin{enumerate}
\item $m(x,x,y)=y$, for all $x,y\in M$;
\item $m(x_1,x_2,x_3)=m(x_{\s(1)},x_{\s(2)},x_{\s(3)})$, for all $x_1,x_2,x_3\in M$ and $\s\in{\rm Sym}(3)$;
\item $m(m(x,y,z),y,w)=m(x,y,m(z,y,w))$.
\end{enumerate}
Median algebras were originally introduced in order theory as a common generalisation of dendrites and lattices. They have been extensively studied in relation to semi-lattices (see e.g.\ \cite{Sholander,Isbell,Bandelt-Hedlikova}) and, more recently, in more geometrical terms because of their connections to the geometry of ${\rm CAT}(0)$ cube complexes and mapping class groups (e.g.\ in \cite{Chepoi,Roller,CDH,Bow-cm,Bow2}). 

A map $\varphi\colon M\ra N$ between median algebras is a \emph{median morphism} if it satisfies the equality $\varphi(m(x,y,z))=m(\varphi(x),\varphi(y),\varphi(z))$ for all $x,y,z\in M$. We denote by $\aut M$ the group of median automorphisms of $M$. 

If $X$ is a median space and $m\colon X^3\ra X$ is the map given by~(\ref{median space defn}), then the pair $(X,m)$ is a median algebra. Every isometric action $G\acts X$ gives an action by median automorphisms on $(X,m)$, but the converse does not hold (as already mentioned, every homeomorphism of $\R$ preserves its usual median-algebra structure). 

The \emph{rank} $\rk M$, is the supremum of $k\in\N$ for which $M$ contains a median subalgebra isomorphic to $\{0,1\}^k$ (the vertex set of a $k$--cube). For connected median spaces, this is equivalent to the definition of rank given earlier in this introduction (see Remark~\ref{definitions of rank}).

A subset $C\cu M$ is \emph{convex} if we have $m(x,y,z)\in C$ for all $x,y\in C$ and $z\in M$. If $M$ is the underlying median algebra of a geodesic median space $X$, then a subset $C$ is convex exactly when it contains all geodesics joining two of its points.

A \emph{halfspace} is a nonempty convex subset $\mf{h}\cu M$ such that its complement $\mf{h}^*:=M\setminus\mf{h}$ is also convex and nonempty. A \emph{wall} is an unordered pair $\{\mf{h},\mf{h}^*\}$, where $\mf{h}$ is a halfspace. A wall $\mf{w}=\{\mf{h},\mf{h}^*\}$ \emph{separates} two points $x,y\in M$ if $\mf{h}\cap\{x,y\}$ is a singleton. We write $\mscr{W}(x|y)$ for the set of walls that separate $x$ and $y$, which is nonempty as soon as $x\neq y$ \cite[Theorem~2.7]{Roller}.

\begin{defn}\label{minset defn}
Let $M$ be a median algebra. Consider $g\in\aut M$.
\begin{enumerate}
\item The \emph{minimal set} $\Min(g)$ is the set of points $x\in M$ such that the sets $\mscr{W}(g^nx|g^{n+1}x)$ are pairwise disjoint for $n\in\Z$.
\item We say that $g$ is \emph{semisimple} if $\Min(g)\neq\emptyset$.
\end{enumerate}
\end{defn}

When $X$ is a geodesic median space and $g\in\isom X$, a point $x\in X$ lies in $\Min(g)$ if and only if either $gx=x$ or $x$ lies on an axis of $g$.

\begin{defn}\label{ess/min defn}
Let $M$ be a median algebra. An action by median automorphisms $G\acts M$ is:
\begin{enumerate}
\item \emph{essential} if, for every halfspace $\mf{h}\cu M$, there exists $g\in G$ such that $g\mf{h}\subsetneq\mf{h}$;
\item \emph{minimal} if there exists no proper, $G$--invariant, convex subset of $M$; 
\item \emph{without wall inversions} if there do not exist $\mf{h}\in\mscr{H}(M)$ and $g\in G$ such that $g\mf{h}=\mf{h}^*$.
\end{enumerate}
\end{defn}

Perhaps counterintuitively, any action $G\acts M$ that originates from an isometric $G$--action on a \emph{connected} finite-rank median space is automatically without wall inversions (see Remark~\ref{connected inversions}). 

We are now ready to state the main result of this paper. For a more canonical (though more technical) result that is roughly equivalent to Theorem~\ref{cores intro}, we refer the reader to Theorem~\ref{abstract core}.

\begin{thmintro}\label{cores intro}
Let $M$ be a finite-rank median algebra. Let $G\acts M$ be an action by median automorphisms without wall inversions. Then:
\begin{enumerate} 
\item $G\acts M$ is essential if and only if $G\acts M$ is minimal;
\item if $G$ is finitely generated, there exists a nonempty, $G$--invariant, convex subset $C\cu M$ such that $G\acts C$ is essential.
\end{enumerate}
\end{thmintro}

When $M$ is the vertex set of a $\CAT$ cube complex $X$, the subset $C$ is $G$--equivariantly isomorphic to the \emph{essential core} of Caprace and Sageev \cite[Subsection~3.3]{CS}. Thus, Theorem~\ref{cores intro} implies that the essential core embeds into $X$ in this case (cf.\ Propositions~3.5 and~3.12 in \cite{CS}). 

Note that part~(2) of Theorem~\ref{cores intro} can fail when $G$ is not finitely generated, and this already happens for isometric actions on simplicial trees (Example~\ref{wollmilchsau}).

The set $C$ in part~(2) of Theorem~\ref{cores intro} is allowed to be a single point, in which case the action $G\acts M$ has a global fixed point. When $G\simeq\Z$, the set $C$ has a very special structure (it is an interval with endpoints in the \emph{zero-completion} of $M$, an analogue of the \emph{Roller boundary} of a cube complex). From this, one can deduce the following result, which, in turn, implies Corollary~\ref{main median semisimple}.

\begin{corintro}\label{ss cor intro}
Let $M$ be a finite-rank median algebra. If, for $g\in\aut M$, the action $\langle g\rangle\acts M$ is without wall inversions, then $g$ is semisimple.
\end{corintro}

We will also deduce the following result, which, along with Theorem~\ref{cores intro}, implies Corollary~\ref{flattorus cor}.

\begin{corintro}\label{corintro polycyclic}
Let $G\acts X$ be a minimal, isometric action without wall inversions on a finite-rank median space. Let $H\leq G$ be a commensurated polycyclic subgroup. Then $X$ has a $G$--invariant product splitting $X=F\x P$ such that:
\begin{itemize}
\item $H\acts P$ fixes a point and $H\acts F$ factors through a virtually abelian group;
\item $F$ isometrically embeds in $(\R^r,d_{\ell^1})$ for $r=\rk X$. 
\end{itemize}
\end{corintro}

A very similar result for actions on (possibly infinite-dimensional) $\CAT$ cube complexes was recently obtained by Genevois (see Theorem~1.1 and Theorem~1.2 in \cite{Genevois-flattorus}).

It would be interesting to see if Corollary~\ref{corintro polycyclic} admits an analogue for general actions by median automorphisms on finite-rank median algebras. There are significant complications that arise in this setting (above all, the failure of Corollary~\ref{core splitting}), but I was not able to find a counterexample. 

\medskip
{\bf Paper Outline.} Section~\ref{prelims sect} contains some basic facts on median algebras. We also obtain a novel extension of Helly's lemma to infinite families of halfspaces (Lemma~\ref{nonempty intersection criterion}), which will prove hugely useful for many of the results in this paper. In Subsection~\ref{wollmilchsau sect}, we describe an example showing that many results of this paper fail for infinitely generated groups.

Section~\ref{main sect} constitutes the heart of this work. After some preliminary facts, we introduce the \emph{convex core} in Subsection~\ref{convex cores sect}. This will be our main tool throughout the paper. Theorem~\ref{abstract core}, proving that cores are nonempty, can also be viewed as our main result, having Theorem~\ref{cores intro} as one of its consequences. The two parts of Theorem~\ref{cores intro} are proved as Proposition~\ref{H=H_1} and Lemma~\ref{essential vs minimal cor} at the end of Subsection~\ref{convex cores sect}. Corollary~\ref{ss cor intro} is obtained in Subsection~\ref{ss sect}. 

In Subsection~\ref{non-transverse sect}, we study \emph{non-transverse} automorphisms of median algebras. The main results are Proposition~\ref{gate-convex C bar} and Proposition~\ref{cores of subalgebras}, which will prove useful in \cite{Fio10a}. 

In Section~\ref{compatible metrics sect}, we prove Corollary~\ref{main median semisimple} and refine a few results in the presence of a median metric (see in particular Corollary~\ref{core splitting} and Proposition~\ref{core of g when metric}). Finally, Section~\ref{polycyclic sect} is devoted to the proof of Corollaries~\ref{flattorus cor} and~\ref{corintro polycyclic}. 

\medskip
{\bf Acknowledgments.} I would like to thank Anthony Genevois for conversations related to Example~\ref{BS example}, Ric Wade for pointing me to \cite[Example~II.6]{Gaboriau-Levitt}, and the referee for their helpful comments and, especially, for suggesting Example~\ref{wollmilchsau}. 

I am grateful to Ursula Hamenst\"adt and the Max Planck Institute for Mathematics in Bonn for their hospitality and financial support while part of this work was being completed.

\tableofcontents

\section{Preliminaries.}\label{prelims sect}

\subsection{Some classical facts.}

Several fundamental notions in the study of median algebras were already defined in the Introduction. In this subsection, we collect some additional terminology and notation, along with a few basic facts. For a more detailed introduction to median algebras, see e.g.\ \cite[Sections~2--4]{CDH}, \cite[Sections~4--6]{Bow-cm} and \cite[Section~2]{Fio1}.

Let $M$ be a median algebra. Given two subsets $A,B\cu M$, we use the notation:
\begin{align*}
\mscr{H}(A|B):=&\{\mf{h}\in\mscr{H}(M) \mid A\cu\mf{h}^*,\ B\cu\mf{h}\}, \\
\mscr{W}(A|B):=&\{\{\mf{h},\mf{h}^*\}\in\mscr{W}(M) \mid \mf{h}\in\mscr{H}(A|B)\cup\mscr{H}(B|A)\}, \\
\mscr{H}_A(M):=&\{\mf{h}\in\mscr{H}(M) \mid \mf{h}\cap A\neq\emptyset,\ \mf{h}^*\cap A\neq\emptyset\}.
\end{align*}
Given a subset $\mc{H}\cu\mscr{H}(M)$, we write $\mc{H}^*:=\{\mf{h}^*\mid\mf{h}\in\mc{H}\}$.

If $\mf{w}=\{\mf{h},\mf{h}^*\}$ is a wall, we say that $\mf{h}$ and $\mf{h}^*$ are its two \emph{sides} and that $\mf{w}$ \emph{bounds} $\mf{h}$ and $\mf{h}^*$. Halfspaces $\mf{h}$ and $\mf{k}$ are \emph{transverse} if each of the four intersections $\mf{h}\cap\mf{k}$, $\mf{h}^*\cap\mf{k}$, $\mf{h}\cap\mf{k}^*$, $\mf{h}^*\cap\mf{k}^*$ is nonempty. Two walls are \emph{transverse} if they bound transverse halfspaces. If $\mc{U}$ and $\mc{V}$ are sets of halfspaces or walls, we say that $\mc{U}$ and $\mc{V}$ are \emph{transverse} if every element of $\mc{U}$ is transverse to every element of $\mc{V}$. 

\begin{rmk}[Definitions of rank]\label{definitions of rank}
In the introduction, we defined $\rk M$ as the largest $k$ such that $M$ has a median subalgebra isomorphic to $\{0,1\}^k$. When $\rk M$ is finite, this coincides with the maximal cardinality of a set of pairwise-transverse walls of $M$ \cite[Proposition~6.2]{Bow-cm}.

If $X$ is a connected median space, then the rank of the underlying median algebra also coincides with the supremum of the topological dimensions of the compact subsets of $X$. One inequality follows from Theorem~2.2 and Lemma~7.6 in \cite{Bow-cm}, while the other from \cite[Proposition~5.6]{Bow4}. 
\end{rmk}

A \emph{pocset} is a triple $(P,\preceq,\ast)$, where $(P,\preceq)$ is a poset and $\ast$ is an order-reversing involution. Given pocsets $P,P'$, a map $f\colon P\ra P'$ is a \emph{morphism of pocsets} if, for all $x,y\in P$ with $x\leq y$, we have $f(x)\leq f(y)$ and $f(x)^*=f(x^*)$. For every median algebra $M$, the triple $(\mscr{H}(M),\cu,\ast)$ is a pocset. 

An \emph{ultrafilter} is a subset $\s\cu\mscr{H}(M)$ such that the halfspaces in $\s$ pairwise intersect and $\s\cap\{\mf{h},\mf{h}^*\}$ is a singleton for every $\mf{h}\in\mscr{H}(M)$. For every $x\in M$, the set of all halfspaces containing $x$ is an ultrafilter, which we denote by $\s_x$.

Convex subsets were defined in the introduction. In this paragraph, we recall two important properties that they enjoy. First, Helly's lemma: every finite family of pairwise-intersecting convex subsets has nonempty intersection \cite[Theorem~2.2]{Roller}. Second, if $C_1,C_2\cu M$ are disjoint convex subsets, then $\mscr{W}(C_1|C_2)\neq\emptyset$ \cite[Theorem~2.8]{Roller}.

\begin{rmk}\label{median homo rmk}
Consider median algebras $M$ and $N$, and a median morphism $\phi\colon M\ra N$. If $C_1\cu M$ and $C_2\cu N$ are convex, then $\phi^{-1}(C_2)$ in $M$, and $\phi(C_1)$ is convex in $\phi(N)$.

When $\phi$ is surjective, we obtain a well-defined map $\phi^*\colon\mscr{H}(N)\ra\mscr{H}(M)$ given by $\phi^*(\mf{h})=\phi^{-1}(\mf{h})$ (surjectivity is needed to avoid empty preimages). The map $\phi^*$ is an injective morphism of pocsets and it preserves transversality. The image of $\phi^*$ is exactly the set of those halfspaces $\mf{h}\in\mscr{H}(M)$ for which $\phi(\mf{h})\cap\phi(\mf{h}^*)=\emptyset$.
\end{rmk}

A subset $C\cu M$ is \emph{gate-convex} if there exists a \emph{gate-projection} $\pi\colon M\ra C$, i.e.\ a map that satisfies $m(x,\pi(x),y)=\pi(x)$ for all $x\in M$ and $y\in C$. When they exist, gate-projections are unique and they are median morphisms. Moreover, $\mscr{H}(x|\pi(x))=\mscr{H}(x|C)$ for every $x\in M$. Every gate-convex subset is convex, and every convex subset is a median subalgebra. 

\begin{rmk}\label{halfspaces of subsets} 
\begin{enumerate}
\item[]
\item If $S\cu M$ is a subalgebra, we have a map $\res_S\colon\mscr{H}_S(M)\ra\mscr{H}(S)$ given by 
$\res_S(\mf{h})=\mf{h}\cap S$. This is a morphism of pocsets and, by \cite[Lemma~6.5]{Bow-cm}, it is a surjection.
\item If $C\cu M$ is convex, then $\res_C$ is injective and it preserves transversality. In this case, the sets $\mscr{H}(C)$ and $\mscr{H}_C(M)$ are naturally identified.

In order to see this, observe that, if $\mf{h},\mf{k}\in\mscr{H}_C(M)$ satisfy $\mf{h}\cap\mf{k}\neq\emptyset$, then Helly's lemma yields 
$(\mf{h}\cap C)\cap(\mf{k}\cap C)\neq\emptyset$. Note that $\mf{h}=\mf{k}$ if and only if $\mf{h}\cap\mf{k}^*$ and 
$\mf{h}^*\cap\mf{k}$ 
are empty.
\item If $C\cu M$ is gate-convex with gate-projection $\pi_C$, we have, in the notation of Remark~\ref{median homo rmk}: $\res_C\o\pi_C^*=\id_{\mscr{H}(C)}$ and $\pi_C^*\o\res_C=\id_{\mscr{H}_C(M)}$.
\end{enumerate}
\end{rmk}

Given points $x,y\in M$, the set
\[I(x,y):=\{z\in M \mid m(x,y,z)=z\}\]
is called the \emph{interval} with endpoints $x$ and $y$. This is a gate-convex subset of $M$ with gate-projection given by the map $z\mapsto m(x,y,z)$. Nonempty intersections of intervals are again intervals.

For all $x,y,z\in M$, we have $I(x,y)\cap I(y,z)\cap I(z,x)=\{m(x,y,z)\}$. The median $m(x,y,z)$ is also the unique point with the property that $m(x,y,z)$ lies in a halfspace $\mf{h}\in\mscr{H}(M)$ if and only if $\mf{h}$ contains at least two of the three points $x,y,z$. We will implicitly rely on the latter characterisation of the median many times throughout the paper. However, we prove explicitly here the following identities, which are a consequence of this principle and are needed in the proof of Lemma~\ref{ss lemma}.

\begin{lem}\label{lem:identities}
For points $a,b,c,x\in M$ and $a_0,\dots,a_k\in M$ with $k\geq 2$, we have:
\begin{align*}
\mscr{W}(a|b,c)&=\mscr{W}(a \mid m(a,b,c)), \\
\mscr{W}(a_1,\dots,a_k|x,a_0)&=\mscr{W}(m(a_1,a_2,x),m(a_2,a_3,x),\dots, m(a_{k-1},a_k,x) \mid x, m(a_0,a_1,x)).
\end{align*}
\end{lem} 
\begin{proof}
The first identity is immediate from the fact that the median of three points lies in a halfspace $\mf{h}$ if and only if $\mf{h}$ contains at least two of the three points. The second identity is obtained by a repeated application of the first, namely the fact that for each $1\leq i \leq k-1$ we have:
\begin{align*}
\mscr{W}(m(a_i,a_{i+1},x) \mid x)&=\mscr{W}(a_i,a_{i+1}| x), & \mscr{W}(a_1 \mid m(a_0,a_1,x))&=\mscr{W}(a_1| x,a_0).
\end{align*}
\end{proof}

We conclude this subsection with a few basic observations on group actions on median algebras.

\begin{rmk}\label{gate-convex inversions}
Let $G\acts M$ be an action by median automorphisms with no wall inversions. If $C\cu M$ is a $G$--invariant convex subset, then $G\acts C$ has no wall inversions. This is because, as observed in Remark~\ref{halfspaces of subsets}, every halfspace of $C$ is of the form $\mf{h}\cap C$ for a \emph{unique} $\mf{h}\in\mscr{H}(M)$. Thus, if $g\mf{h}\cap C=\mf{h}^*\cap C$ for some $g\in G$, we must have $g\mf{h}=\mf{h}^*$.
\end{rmk}

\begin{lem}\label{from finite orbit to finite cube 1}
Let $G\acts M$ be an action by median automorphisms with a finite orbit. Then there exists a $G$--invariant median subalgebra $C\cu M$ such that $C\simeq\{0,1\}^k$ for some $0\leq k<+\infty$.
\end{lem}
\begin{proof}
Let $x\in M$ be a point such that $G\cdot x$ is finite. The median subalgebra of $M$ generated by $G\cdot x$ is $G$--invariant and, by \cite[Lemma~4.2]{Bow-cm}, it is finite. It was observed in \cite[Section~3]{Bow4} that every finite median algebra is isomorphic to the $0$--skeleton of a finite $\CAT$ cube complex. If $X$ is a finite $\CAT$ cube complex with $X^{(0)}$ isomorphic to $\langle G\cdot x\rangle$, the induced action $G\acts X^{(0)}$ preserves the median of $X$, hence it extends to an action by cubical automorphisms on the entire $X$. Since $X$ is a compact $\CAT$ space, it has a unique barycentre $p\in X$. Since the $G$--action preserves the $\CAT$ metric on $X$, the point $p$ is fixed by $G$. Thus $G$ leaves invariant the $0$--skeleton of the unique open cube of $X$ containing $p$. This provides the required median subalgebra $C\cu M$.
\end{proof}

The following examples were mentioned in the Introduction.

\begin{ex}\label{L^1 not semisimple ex}
There are isometries of $X=L^1(\R)$ whose translation distance is not realised. Note that $X$ is a complete, geodesic, infinite-rank median space with the metric induced by its norm.

As far as I know, the following is the quickest way to see this. There exist a standard Borel space $\Om$ and an infinite, $\s$--finite, atomless measure $\nu$ such that $\H^2$ embeds in $Y=L^1(\Om,\nu)$ isometrically and equivariantly with respect to an embedding $\iota\colon\isom\H^2\hookrightarrow\isom Y$ (see e.g.\ Example~3.7 and Proposition~3.14 in \cite{CDH}). Since $\Om$ is standard, the space $L^1(\Om,\nu)$ is isometric to $L^1(\R)$. 

If $g$ is a parabolic isometry of $\H^2$, then $\langle g\rangle$ acts with unbounded orbits on $\H^2$, so $\iota(g)$ acts with unbounded orbits on $Y$. In particular, $\iota(g)$ does not fix a point of $Y$. However, since the translation length of $g$ in $\H^2$ is zero, there exist points of $Y$ that are moved arbitrarily little by $\iota(g)$.
\end{ex}

\begin{ex}\label{BS example}
Let $G$ be the group $BS(1,n)=\langle a,b \mid aba^{-1}=b^n\rangle$ with $n\geq 2$. Let $G\acts T$ be the Bass--Serre tree corresponding to the HNN splitting evident from this presentation. Consider the action $G\acts\R$ where $b$ is a translation of length $1$ and $a$ is a homothety of factor $n$ fixing the origin. 

Thus, $G\acts\R$ is an action by homotheties and $G\acts T$ is an action by isometries. The product $T\x\R$ is a rank--$2$ median space, and the diagonal action $G\acts T\x\R$ is by homeomorphisms that are also automorphisms of the median-algebra structure. It is straightforward to check that the latter action is cocompact, free, and metrically proper (though not uniformly proper).

As mentioned in the introduction, \cite[Theorem~A]{Fio2} implies that $G$ does not admit any free \emph{isometric} actions on finite-rank median spaces. However, $G$ does act properly, freely and by isometries on \emph{infinite-rank} median spaces, since $G$ has the Haagerup property \cite[Theorem~1.2]{CDH}.
\end{ex}

\subsection{Restriction quotients.}\label{restriction quotients sect}

Let $M$ be a median algebra with an action $G\acts M$ by median automorphisms. In the special case when $M$ is the vertex set of a $\CAT$ cube complex, the notion of \emph{restriction quotient} was introduced in \cite[Subsection~2.3]{CS}.

Consider a $G$--invariant subset $\mc{U}\cu\mscr{W}(M)$. Given $x,y\in M$, we write $x\sim_{\mc{U}}y$ if $\mscr{W}(x|y)\cap\mc{U}=\emptyset$. Observe that, for all $x,y,z\in M$ and $x',y',z'\in M$, we have:
\begin{align*}
\mscr{W}(x|y)&\cu\mscr{W}(x|z)\cup\mscr{W}(z|y),& \mscr{W}(m(x,y,z)|m(x',y',z'))&\cu\mscr{W}(x|x')\cup\mscr{W}(y|y')\cup\mscr{W}(z|z').
\end{align*}
We conclude that $\sim_{\mc{U}}$ is an equivalence relation and that the $\sim_{\mc{U}}$--equivalence class of $m(x,y,z)$ only depends on the $\sim_{\mc{U}}$--equivalence classes of $x,y,z$. Thus, the quotient $M/\sim_{\mc{U}}$ naturally inherits a structure of median algebra, which we denote by $M(\mc{U})$. Since $\mc{U}$ is $G$--invariant, so is $\sim_{\mc{U}}$, and $M(\mc{U})$ is again endowed with a $G$--action by median automorphisms.

\begin{defn}
We say that $M(\mc{U})$ is the \emph{restriction quotient} of $M$ associated to $\mc{U}$.
\end{defn}

The quotient projection $\pi_{\mc{U}}\colon M\ra M(\mc{U})$ is a $G$--equivariant, surjective median morphism. Remark~\ref{median homo rmk} shows that the projection $\pi_{\mc{U}}$ induces an injective, $G$--equivariant morphism of pocsets $\pi_{\mc{U}}^*\colon\mscr{H}(M(\mc{U}))\ra\mscr{H}(M)$. Since $\pi_{\mc{U}}^*$ preserves transversality, this shows that $\rk M(\mc{U})\leq\rk M$. 

We say that a subset $A\cu M$ is \emph{$\sim_{\mc{U}}$--saturated} if it is a union of $\sim_{\mc{U}}$--equivalence classes; equivalently $A=\pi_{\mc{U}}^{-1}(\pi_{\mc{U}}(A))$. The image of $\pi_{\mc{U}}^*$ is precisely the set of $\sim_{\mc{U}}$--saturated halfspaces. In general, this will be larger than the set of halfspaces bounded by the walls in $\mc{U}$.

\subsection{The fundamental lemma.}

Let $M$ be a median algebra of rank $r<+\infty$.

Given a subset $\mc{H}\cu\mscr{H}(M)$, we denote by $\bigcap\mc{H}$ the intersection of all halfspaces that it contains. We say that a subset $\mscr{C}\cu\mscr{H}(M)$ is a \emph{chain} of halfspaces if it is totally ordered by inclusion. 

The following simple fact is a key ingredient of many results in this paper. It can be viewed as an extension of Helly's lemma to infinite sets of halfspaces.

\begin{lem}\label{nonempty intersection criterion}
Let $\mc{H}\cu\mscr{H}(M)$ be a set of pairwise-intersecting halfspaces such that any chain in $\mc{H}$ admits a lower bound in $\mc{H}$. Then the intersection of all halfspaces in $\mc{H}$ is nonempty.
\end{lem}
\begin{proof}
Pick any point $x\in M$ and let $\mc{H}_x\cu\mc{H}$ be the subset of those halfspaces that do not contain $x$. We begin by finding a point $y\in M$ that lies in every element of $\mc{H}_x$. 

Given $\mf{h},\mf{k}\in\mc{H}_x$, we have $\mf{h}\cap\mf{k}\neq\emptyset$ by our hypothesis on $\mc{H}$, and $x\in\mf{h}^*\cap\mf{k}^*$ by definition of $\mc{H}_x$. Hence either $\mf{h}\cu\mf{k}$, or $\mf{k}\cu\mf{h}$, or $\mf{h}$ and $\mf{k}$ are transverse. Since $M$ has finite rank $r$, Dilworth's lemma allows us to partition $\mc{H}_x=\mscr{C}_1\sqcup\dots\sqcup\mscr{C}_k$, where $k\leq r$ and each $\mscr{C}_i$ is totally ordered by inclusion.  Let $\mf{h}_1,\dots,\mf{h}_k\in\mc{H}$ be lower bounds for $\mscr{C}_1,\dots,\mscr{C}_k$. Since the elements of $\mc{H}$ pairwise intersect, Helly's lemma implies that $\mf{h}_1\cap\dots\cap\mf{h}_k\neq\emptyset$. We pick the point $y$ in this intersection.

Now, let $\mc{H}_{x,y}\cu\mc{H}$ be the subset of halfspaces that contain exactly one point from the set $\{x,y\}$. Note that $\mc{H}_{x,y}\cu\mscr{H}(x|y)\sqcup\mscr{H}(y|x)$. Hence, again by Dilworth's lemma, we can partition $\mc{H}_{x,y}$ into finitely many chains and there exists a point $z\in\bigcap\mc{H}_{x,y}$. 

Setting $w=m(x,y,z)$, we finally show that $w\in\bigcap\mc{H}$. Suppose for the sake of contradiction that $w\in\mf{h}^*$ for some $\mf{h}\in\mc{H}$. Then at least two of the three points $x,y,z$ must lie in $\mf{h}^*$. By our choice of $y$, there does not exist $\mf{h}\in\mc{H}$ with $\{x,y\}\cu\mf{h}^*$. By our choice of $z$, there does not exist $\mf{h}\in\mc{H}$ with $\mf{h}\in\mscr{H}(x,z|y)\sqcup\mscr{H}(y,z|x)$. This is the desired contradiction.
\end{proof}

\begin{rmk}\label{bigcap chain halfspace}
If $\mscr{C}$ is a chain of halfspaces and $\bigcap\mscr{C}\neq\emptyset$, then $\bigcap\mscr{C}$ is a halfspace. Indeed, both $\bigcap\mscr{C}$ and its complement are convex (ascending unions of convex sets are convex).
\end{rmk}

If $\mscr{C}$ is a chain of halfspaces, we say that a subset $\mscr{C}'\cu\mscr{C}$ is \emph{cofinal} if every halfspace in $\mscr{C}$ contains a halfspace in $\mscr{C}'$. In this case, we have $\bigcap\mscr{C}'=\bigcap\mscr{C}$.

\begin{lem}\label{gate-convexity criterion}
A convex subset $C\cu M$ is gate-convex if and only if there does not exist a chain $\mscr{C}\cu\mscr{H}_C(M)$ such that $\bigcap\mscr{C}$ is nonempty and disjoint from $C$.
\end{lem}
\begin{proof}
Suppose that $C$ is gate-convex. If $\mscr{C}\cu\mscr{H}_C(M)$ is a chain and $x\in\bigcap\mscr{C}$, then the gate-projection of $x$ to $C$ also lies in $\bigcap\mscr{C}$. Thus either $\bigcap\mscr{C}=\emptyset$ or $C\cap\bigcap\mscr{C}\neq\emptyset$.

Conversely, suppose that $C$ is convex and that there does not exist a chain $\mscr{C}\cu\mscr{H}_C(M)$ such that $\bigcap\mscr{C}$ is nonempty and disjoint from $C$. We define a map $\pi\colon M\ra C$ as follows.

Let $\s_C\cu\mscr{H}(M)$ be the set of halfspaces containing $C$. For every $x\in M$, consider the set:
\[\mc{H}_x=\s_C\cup(\s_x\setminus\mscr{H}(C|x)).\]
Note that any two halfspaces $\mf{h},\mf{k}\in\mc{H}_x$ have nonempty intersection. This is clear if they both lie in $\s_C$ or $\s_x$; if instead $\mf{h}\in\s_C$ and $\mf{k}\in\s_x\setminus\mscr{H}(C|x)$, this follows from the fact that $C\cu\mf{h}$ and $\mf{k}\cap C\neq\emptyset$. Also note that, for every $\mf{j}\in\mscr{H}(M)$, either $\mf{j}$ or $\mf{j}^*$ lies in $\mc{H}_x$. Indeed, assuming that $x\in\mf{j}$, either $\mf{j}\cap C\neq\emptyset$ and $\mf{j}\in\s_x\setminus\mscr{H}(C|x)\cu\mc{H}_x$, or $C\cu\mf{j}^*$ and $\mf{j}^*\in\s_C\cu\mc{H}_x$. 

In conclusion, $\mc{H}_x$ is an ultrafilter. In addition, let us show that every chain $\mscr{C}\cu\mc{H}_x$ admits a lower bound in $\mc{H}_x$. First, $\mscr{C}$ contains a cofinal subset $\mscr{C}'$ that is contained in either $\s_C$ or $\s_x\setminus\mscr{H}(C|x)$. If $\mscr{C}'\cu\s_C$, then $C\cu\bigcap\mscr{C}'$. If $\mscr{C}'\cu\s_x\setminus\mscr{H}(C|x)$, then $x\in\bigcap\mscr{C}'$ and, by our assumption on $C$, we have $\bigcap\mscr{C}'\cap C\neq\emptyset$. In both cases, since $\bigcap\mscr{C}=\bigcap\mscr{C}'\neq\emptyset$, Remark~\ref{bigcap chain halfspace} implies that $\bigcap\mscr{C}$ is a halfspace, and we have $\bigcap\mscr{C}\in\mc{H}_x$.

Now, we can apply Lemma~\ref{nonempty intersection criterion} and conclude that $\bigcap\mc{H}_x\neq\emptyset$. Since $\mc{H}_x$ is an ultrafilter, we deduce that this intersection consists of a single point $\pi(x)$. This defines the map $\pi\colon M\ra C$.

Since $\s_C\cu\mc{H}_x$, we have $\pi(x)\in C$. It follows that $\mscr{W}(x|C)=\mscr{W}(x|\pi(x))$ for all $x\in M$. Hence $\pi(x)\in I(x,y)$ for every $x\in M$ and $y\in C$, showing that $C$ is gate-convex.
\end{proof}

\subsection{Products.}

Given two median algebras $M_1$ and $M_2$, we denote by $M_1\x M_2$ their \emph{product}. This is the only median algebra with underlying set $M_1\x M_2$ such that the coordinate projections to $M_1$ and $M_2$ are median morphisms. 

We will need the following result. When $M$ is the vertex set of a $\CAT$ cube complex, this is \cite[Lemma~2.5]{CS}. When $M$ is the underlying median algebra of a complete median space, compare \cite[Proposition~2.10]{Fio2}\footnote{In particular, the equivalence of parts~(1) and~(2) in \cite[Proposition~2.10]{Fio2} does not require a metric, or measurability of the partition. By contrast, we stress that the equivalence with part~(3) of \cite[Proposition~2.10]{Fio2} does require a \emph{complete} metric. Without it, a counterexample is provided by $[0,1]^2\setminus\{(0,0)\}$ with the $\ell^1$ metric.}. In full generality, we will rely on Lemma~\ref{nonempty intersection criterion}.

\begin{lem}\label{product median algebras}
For a finite-rank median algebra $M$, the following are equivalent:
\begin{enumerate}
\item $M$ splits as a nontrivial product of median algebras $M_1\x M_2$;
\item there exists a nontrivial partition $\mscr{W}(M)=\mc{W}_1\sqcup\mc{W}_2$ such that $\mc{W}_1$ and $\mc{W}_2$ are transverse.
\end{enumerate}
When this happens, each set $\mc{W}_i$ is naturally identified with $\mscr{W}(M_i)$.
\end{lem}
\begin{proof}
First, suppose that $M=M_1\x M_2$. Let $\pi_i$ denote the factor projections. For all $x,y\in M$, we have $I(x,y)=I(\pi_1(x),\pi_1(y))\x I(\pi_2(x),\pi_2(y))$. Thus, every convex set $C\cu M$ is necessarily of the form $C_1\x C_2$, where each $C_i\cu M_i$ is convex. It follows that every halfspace of $M$ is either of the form $\mf{h}_1\x M_2$ with $\mf{h}_1\in\mscr{H}(M_1)$, or of the form $M_1\x\mf{h}_2$ with $\mf{h}_2\in\mscr{H}(M_2)$. This proves the implication (1)$\Ra$(2) and the final statement of the lemma. 

We are left to prove that (2)$\Ra$(1). The partition $\mscr{W}(M)=\mc{W}_1\sqcup\mc{W}_2$ determines a nontrivial, transverse partition $\mscr{H}(M)=\mc{H}_1\sqcup\mc{H}_2$, where $\mc{H}_i$ is the set of halfspaces associated to $\mc{W}_i$. Given $x\in M$, recall that $\s_x$ denotes the set of halfspaces containing $x$. Setting $\s_x^i:=\s_x\cap\mc{H}_i$, we obtain a transverse partition $\s_x=\s_x^1\sqcup\s_x^2$. 

Let $A_x,B_x\cu M$ denote the intersection of all halfspaces in $\s_x^2$ and $\s_x^1$, respectively. These are convex subsets with $A_x\cap B_x=\{x\}$. Given $y\in A_x$ and $z\in B_x$, the intersection of all halfspaces in $\s_y^1\sqcup\s_z^2$ is nonempty, since this set satisfies the hypotheses of Lemma~\ref{nonempty intersection criterion}. Moreover, the intersection must consist of a single point, since $\s_y^1\sqcup\s_z^2$ contains a side of every wall of $M$. We denote this point by $\phi(y,z)$. This defines a map $\phi\colon A_x\x B_x\ra M$. 

If $(y,z),(y',z')\in A_x\x B_x$ are distinct, then either $y$ and $y'$ are separated by a wall in $\mc{W}_1$, or $z$ and $z'$ are separated by a wall in $\mc{W}_2$. The same wall will separate $\phi(y,z)$ and $\phi(y',z')$, hence $\phi$ is injective. Given $w\in M$, the argument of the previous paragraph also shows that there exists a point $y_w$ lying in the intersection of the halfspaces in $\s_w^1\sqcup\s_x^2$; note that, in particular, we have $y_w\in A_x$. Similarly, there exists $z_w\in B_x$ lying in every halfspace of the set $\s_x^1\sqcup\s_w^2$. Observing that $w=\phi(y_w,z_w)$, we conclude that $\phi$ is also surjective. Thus, $\phi$ is a bijection.

Finally, we show that $\phi$ is a morphism of median algebras. If this were not the case, there would exist $y_1,y_2,y_3\in A_x$ and $z_1,z_2,z_3\in B_x$ such that the points
\begin{align*}
&m(\phi(y_1,z_1),\phi(y_2,z_2),\phi(y_3,z_3)) &\text{ and }& &\phi(m(y_1,y_2,y_3),m(z_1,z_2,z_3))
\end{align*}
are distinct. Hence there would exist a halfspace $\mf{h}\in\mscr{H}(M)$ containing the latter point, but not the former. Without loss of generality $\mf{h}\in\mc{H}_1$. Then $m(y_1,y_2,y_3)\in\mf{h}$, while at least two of the three points $\phi(y_1,z_1),\phi(y_2,z_2),\phi(y_3,z_3)$ must lie in $\mf{h}^*$. Since $\mf{h}\in\mc{H}_1$, the latter implies that at least two among $y_1,y_2,y_3$ lie in $\mf{h}^*$, contradicting the fact that $m(y_1,y_2,y_3)\in\mf{h}$. 
\end{proof}

We say that $M$ is \emph{irreducible} if it is not isomorphic to a nontrivial product $M_1\x M_2$. It is standard to deduce the following from Lemma~\ref{product median algebras} (cf.\ \cite[Proposition~2.6]{CS} or \cite[Proposition~2.12]{Fio2}).

\begin{cor}\label{product cor}
Let $M$ be a median algebra of rank $r<+\infty$. Then, there exist irreducible median algebras $M_1,\dots,M_k$ (unique up to permutation) such that $M\simeq M_1\x\dots\x M_k$ and $k\leq r$. The group $\aut M$ preserves this splitting, possibly permuting the factors.
\end{cor}

\subsection{An eierlegende Wollmilchsau.}\label{wollmilchsau sect}

For groups that are not \emph{finitely generated}, Corollary~\ref{flattorus cor} and Theorem~\ref{cores intro}(2) fail, and so do essentially all other results of this paper that make a finite generation assumption. As suggested to us by the referee, the following is the shared counterexample for all these results, based on a classical construction due to Serre (see e.g.\ \cite[Theorem~15(3)]{Serre}). 

We discuss each result using the appropriate terminology and notation, which might not have been introduced yet. We encourage the reader to return to this subsection when appropriate.

\begin{ex}\label{wollmilchsau}
Let $G$ be a group that is not finitely generated. 

Choose an enumeration $G=\{g_n\mid n\in\N\}$ and consider the subgroups $G_n:=\langle g_0,\dots,g_n\rangle$. Form a tree $T$ with vertex set $\bigsqcup_{n\in\N} G/G_n$ and edges corresponding to inclusions of cosets of $G_n$ into cosets of $G_{n+1}$ for $n\in\N$. We have an action $G\acts T$ that on $T^{(0)}$ is simply left multiplication of cosets.

Here are the results that fail for this action.
\begin{itemize}
\item Corollary~\ref{flattorus cor}. The group $G$ cannot stabilise a subset $L\cu T$ isometric to a subset of $\R$ (not even when $G$ is abelian). Indeed, $T$ does not contain any subsets isometric to $\R$, so $L$ would have to be isometric to a \emph{proper} subset of $\R$. In particular, $G$ would need to fix $L$ pointwise, although $G$ has no fixed points in $T$.
\item Theorem~\ref{cores intro}(2) and Proposition~\ref{H=H_1}. There are no nonempty, $G$--invariant, convex subsets of $T$ on which $G$ acts essentially (for instance, by Remark~\ref{halfspaces of subsets} and the fact that $\mc{H}_1(G)=\emptyset$).
\item Theorem~\ref{abstract core}, Corollary~\ref{cores intersect subalgebras}. We have $\mscr{H}(T)=\mc{H}_{1/2}(G)\sqcup\mc{H}_{1/2}(G)^*$, so $\mc{C}(G)=\overline{\mc{C}}(G)=\emptyset$.
\item Lemma~\ref{efficient facing}. For every $n$, there is an edge $e_n\cu T$ with stabiliser $G_n$. Let $\mf{h}_n$ be a corresponding halfspace in $\mc{H}_{1/2}(G)$. Since the $G_n$ exhaust $G$, no finite subset $F\cu G$ contains elements $f_n$ such that $\mf{h}_n$ and $f_n\mf{h}_n$ are facing for all $n\in\N$.
\item Lemma~\ref{previously part of the proof}. Every ray in $T$ gives chains in $\mc{H}_{1/2}(G)$ with empty intersection.
\item Lemma~\ref{halfspaces of core}(2). For this, we need to consider the induced $G$--action on the Roller compactification $\overline T=T\sqcup\{\xi\}$. The reader can think of $\xi$ as the point at infinity determined by all rays in $T$ (which are pairwise asymptotic), though this is irrelevant here. 

The median algebra structure on $\overline T$ is completely determined by the fact that $T$ is a convex subset. The $G$--action on $T$ extends to an action by median automorphisms on $\overline T$ fixing $\xi$. The set $\mscr{H}(\overline T)$ is naturally identified with the union $\mscr{H}(T)\sqcup\{\mf{h},\mf{h}^*\}$, where $\mf{h}=\{\xi\}$ and $\mf{h}^*=T$. We have $\mc{C}(G,\overline T)=\overline{\mc{C}}(G,\overline T)=\{\xi\}$, hence $\mscr{H}_{\mc{C}(G)}(\overline T)=\mscr{H}_{\overline{\mc{C}}(G)}(\overline T)=\emptyset$. On the other hand, $\mc{H}_0(G,\overline T)=\overline{\mc{H}}_0(G,\overline T)=\{\mf{h},\mf{h}^*\}\neq\emptyset$, contradicting the conclusion of the lemma.
\item Proposition~\ref{finite orbit new prop}. We have $\mc{H}_1(G)=\emptyset$, but every $G$--orbit in $T$ is infinite.
\end{itemize}
\end{ex}

The only other result of the paper making a finite generation assumption is Remark~\ref{factors through finite group}. For that, a counterexample in the absence of finite generation is given by the left-multiplication action of $\bigoplus_{\N}\Z/2\Z$ on itself, equipped with the coordinate-wise median operator.

\section{Automorphisms of finite-rank median algebras.}\label{main sect}

This section constitutes the heart of the paper. We will define \emph{convex cores} in Subsection~\ref{convex cores sect}, where we also prove Theorem~\ref{cores intro}. We will then study semisimplicity of automorphisms in Subsection~\ref{ss sect}, proving Corollary~\ref{ss cor intro}, which in turn implies Corollary~\ref{main median semisimple}. We conclude the section by studying non-transverse automorphisms in Subsection~\ref{non-transverse sect}.

Throughout Section~\ref{main sect}, we fix a median algebra $M$ of rank $r<+\infty$, a group $G$, and an action $G\acts M$ by median automorphisms.

\subsection{Dynamics of halfspaces.}\label{dynamics sect}

\begin{defn}
We say that two halfspaces $\mf{h},\mf{k}\in\mscr{H}(M)$ are \emph{facing} if $\mf{h}^*\cap\mf{k}^*=\emptyset$ and $\mf{h}\neq\mf{k}^*$.
\end{defn}

The following $G$--invariant subsets of $\mscr{H}(M)$ will play a fundamental role throughout Section~\ref{main sect}. 
\begin{align*}
\mc{H}_1(G):=&\{\mf{h}\in\mscr{H}(M) \mid \exists g\in G \text{ such that } g\mf{h}\subsetneq\mf{h}\} , \\ 
\mc{H}_0(G):=&\{\mf{h}\in\mscr{H}(M) \mid \text{the orbit $G\cdot\mf{h}$ is finite}\}, \\
\mc{H}_{1/2}(G):=&\{\mf{h}\in\mscr{H}(M)\setminus\mc{H}_0(G) \mid \forall g\in G, \text{ $g\mf{h}$ and $\mf{h}$ are either facing, transverse, or equal} \}, \\ 
\overline{\mc{H}}_0(G):=&\{\mf{h}\in\mscr{H}(M) \mid \forall g\in G \text{ either $g\mf{h}\in\{\mf{h},\mf{h}^*\}$ or $g\mf{h}$ and $\mf{h}$ are transverse}\}, \\
\overline{\mc{H}}_{1/2}(G):=&\{\mf{h}\in\mscr{H}(M)\setminus\mc{H}_1(G) \mid \exists g\in G \text{ such that $g\mf{h}$ and $\mf{h}$ are facing}\}.
\end{align*}
We chose this notation by analogy with the terminology of Caprace and Sageev from \cite[Subsection~3.3]{CS}; we will explain the connection in detail in Remark~\ref{explanation of 0,1,1/2}. When it is necessary to specify the median algebra acted upon by $G$, we will also write $\mc{H}_{\bullet}(G,M)$ and $\overline{\mc{H}}_{\bullet}(G,M)$. 

Note that $\mc{H}_0(G)^*=\mc{H}_0(G)$, $\overline{\mc{H}}_0(G)^*=\overline{\mc{H}}_0(G)$ and $\mc{H}_1(G)^*=\mc{H}_1(G)$.

\begin{rmk}\label{0 bar rmk}
Since $\rk M=r$, the $G$--orbit of any halfspace $\mf{h}\in\overline{\mc{H}}_0(G)$ contains at most $2r$ halfspaces. In particular, we have $\overline{\mc{H}}_0(G)\cu\mc{H}_0(G)$ and it follows that $\mc{H}_{1/2}(G)\cu\overline{\mc{H}}_{1/2}(G)$.
\end{rmk}

\begin{rmk}\label{1/2 bar PWI}
The halfspaces in $\overline{\mc{H}}_{1/2}(G)$ pairwise intersect. Indeed, suppose for the sake of contradiction that there exist $\mf{h},\mf{k}\in\overline{\mc{H}}_{1/2}(G)$ with $\mf{h}\cap\mf{k}=\emptyset$. Choose $g_1,g_2\in G$ such that $g_1\mf{h}$ and $\mf{h}$ are facing, and $g_2\mf{k}$ and $\mf{k}$ are facing. Then $g_2\mf{h}\cu g_2\mf{k}^*\subsetneq\mf{k}\cu\mf{h}^*\subsetneq g_1\mf{h}$. Hence $g_1^{-1}g_2\mf{h}\subsetneq\mf{h}$, contradicting the assumption that $\mf{h}\not\in\mc{H}_1(G)$.
\end{rmk}

\begin{rmk}\label{7 possibilities}
Given two halfspaces $\mf{h},\mf{k}\in\mscr{H}(M)$, exactly seven possibilities may occur as to their relative position. We might have $\mf{h}=\mf{k}$, $\mf{h}=\mf{k}^*$, $\mf{h}\subsetneq\mf{k}$, $\mf{k}\subsetneq\mf{h}$, $\mf{h}$ and $\mf{k}$ might be facing, $\mf{h}^*$ and $\mf{k}^*$ might be facing, or, finally, $\mf{h}$ and $\mf{k}$ might be transverse.
\end{rmk}

Since $\overline{\mc{H}}_{1/2}(G)$ is $G$--invariant, the following is immediate from Remarks~\ref{1/2 bar PWI} and~\ref{7 possibilities}:

\begin{rmk}\label{1/2 bar rmk}
If $\mf{h}\in\overline{\mc{H}}_{1/2}(G)$, then there does not exist $g_1\in G$ such that $g_1\mf{h}=\mf{h}^*$, nor does there exist $g_2\in G$ such that $g_2\mf{h}^*$ and $\mf{h}^*$ are facing. In particular, we have $\overline{\mc{H}}_{1/2}(G)\cu\mc{H}_0(G)\sqcup\mc{H}_{1/2}(G)$.
\end{rmk}

\begin{lem}\label{G and h}
We have $G$--invariant partitions:
\begin{align*}
\mscr{H}(M)&=\mc{H}_1(G)\sqcup\mc{H}_0(G)\sqcup\mc{H}_{1/2}(G)\sqcup\mc{H}_{1/2}(G)^*, \\
\mscr{H}(M)&=\mc{H}_1(G)\sqcup\overline{\mc{H}}_0(G)\sqcup\overline{\mc{H}}_{1/2}(G)\sqcup\overline{\mc{H}}_{1/2}(G)^*,
\end{align*}
where $\mc{H}_{1/2}(G)\cu\overline{\mc{H}}_{1/2}(G)\cu\mc{H}_0(G)\sqcup\mc{H}_{1/2}(G)$ and $\overline{\mc{H}}_0(G)\cu\mc{H}_0(G)$.
\end{lem}
\begin{proof}
The second partition is immediate from Remark~\ref{7 possibilities}. Regarding the first, note that the sets $\mc{H}_{1/2}(G)$ and $\mc{H}_{1/2}(G)^*$ are disjoint since, as shown in Remark~\ref{0 bar rmk}, we have $\mc{H}_{1/2}(G)\cu\overline{\mc{H}}_{1/2}(G)$ and $\mc{H}_{1/2}(G)^*\cu\overline{\mc{H}}_{1/2}(G)^*$. Thus it is also clear that the sets $\mc{H}_1(G)$, $\mc{H}_0(G)$, $\mc{H}_{1/2}(G)$ and $\mc{H}_{1/2}(G)^*$ are pairwise disjoint. The fact that these four sets cover $\mscr{H}(M)$, as well as the rest of the lemma, follow from Remarks~\ref{0 bar rmk} and~\ref{1/2 bar rmk}.
\end{proof}

\begin{rmk}\label{G and h rmk}
If $H\leq G$ is a finite-index subgroup, then it is clear that $\mc{H}_0(H)=\mc{H}_0(G)$ and $\mc{H}_1(H)=\mc{H}_1(G)$. It follows that we have $\mc{H}_{1/2}(H)=\mc{H}_{1/2}(G)$. The sets $\overline{\mc{H}}_0(H)$ and $\overline{\mc{H}}_{1/2}(H)$, however, will differ from $\overline{\mc{H}}_0(G)$ and $\overline{\mc{H}}_{1/2}(G)$ in general.
\end{rmk}

When $G=\langle g\rangle$, we simply write $\mc{H}_{\bullet}(g,M)$ or $\mc{H}_{\bullet}(g)$, rather than $\mc{H}_{\bullet}(\langle g\rangle,M)$. In this case, we can give a simpler characterisation of the sets $\mc{H}_{\bullet}(g)$ as follows:
\begin{itemize}
\item $\mf{h}\in\mc{H}_0(g)$ if and only if there exists $n\in\Z\setminus\{0\}$ such that $g^n\mf{h}=\mf{h}$;
\item $\mf{h}\in\mc{H}_1(g)$ if and only if there exists $n\in\Z$ such that $g^n\mf{h}\subsetneq\mf{h}$;
\item $\mf{h}\in\mc{H}_{1/2}(g)$ if and only if, for every $n\in\Z\setminus\{0\}$, the halfspaces $g^n\mf{h}$ and $\mf{h}$ are either facing or transverse.
\end{itemize}

We conclude this subsection with two results on the trivial case where the $G$--orbit of every halfspace of $M$ is finite. They will prove useful in Subsection~\ref{convex cores sect}.

\begin{lem}\label{finite orbit new lem}
If $\mscr{H}(M)=\mc{H}_0(G)$, then all $G$--orbits in $M$ are finite.
\end{lem}
\begin{proof}
We begin with an observation. Consider two distinct points $x,y\in M$. Any two minimal elements of $\mscr{H}(x|y)$ are transverse, and so are any two maximal elements. Thus, there are at most $2r$ among minimal and maximal elements of $\mscr{H}(x|y)$, and a finite-index subgroup $G(x,y)\leq G$ leaves each of them invariant. By Remark~\ref{bigcap chain halfspace}, every halfspace of $\mscr{H}(x|y)$ contains a minimal element of $\mscr{H}(x|y)$ and is contained in a maximal element. Thus, $G(x,y)$ preserves the set $\mscr{H}(x|y)$. Hence, for every $g\in G(x,y)$, we have $\mscr{W}(x|y)=\mscr{W}(x,gx|y,gy)$ and $\mscr{W}(x|gx)=\mscr{W}(x,y|gx,gy)$. It follows that $\mscr{W}(x|y)$ and $\mscr{W}(x|gx)$ are transverse for all $g\in G(x,y)$.

Now, suppose for the sake of contradiction that there is an infinite orbit $G\cdot x\cu M$. Choose $g_0\in G$ with $g_0x\neq x$ and set $G_1:=G(x,g_0x)$. Since $G_1$ has finite index in $G$, there exists $g_1\in G_1$ such that $g_1x\neq x$, and we set $G_2:=G_1\cap G(x,g_1x)$. Iterating, we obtain a chain of finite-index subgroups $G_r\leq\dots\leq G_1\leq G$ and elements $g_i\in G_i$ with $g_ix\neq x$ and $G_{i+1}=G_i\cap G(x,g_ix)$. 

Note that, for $i<j$, the sets $\mscr{W}(x|g_ix)$ and $\mscr{W}(x|g_jx)$ are transverse. Indeed, $g_j$ lies in $G_j$, which is contained in $G(x,g_ix)$.

In conclusion, the sets $\mscr{W}(x|g_0x),\dots,\mscr{W}(x|g_rx)$ are nonempty and pairwise transverse. This violates the assumption that $\rk M=r$.
\end{proof}

\begin{rmk}\label{factors through finite group}
If $\mscr{H}(M)=\overline{\mc{H}}_0(G)$ and $G$ is finitely generated, then $G\acts M$ factors through the action of a finite group. 

Indeed, by Remark~\ref{0 bar rmk}, the $G$--stabiliser of any halfspace of $M$ has index $\leq 2r$ in $G$. Since $G$ is finitely generated, it contains only finitely many subgroups of index $\leq 2r$. Their intersection is a finite-index characteristic subgroup $H\lhd G$ that preserves every halfspace of $M$. Hence $H$ fixes every point of $M$, and the $G$--action factors through the finite group $G/H$.
\end{rmk}

\subsection{Some technical results.}

We gather here a few technical results that we do not deem of much interest on their own. They will be needed in the proofs of some results in later sections, but we prove them here as they only require the terminology of Subsection~\ref{dynamics sect}. The reader can safely skip this subsection and return when they encounter a reference to something proved here.

The following is needed to prove Lemma~\ref{H_1 not saturated}, which is used in the proof of Proposition~\ref{H=H_1}.

\begin{rmk}\label{miracle on H_0}
Consider a chain $\mscr{C}\cu\mc{H}_0(G)$ such that $\bigcap\mscr{C}\in\mscr{H}(M)$. Then $\bigcap\mscr{C}\not\in\mc{H}_1(G)$. 

Indeed, set $\mf{k}=\bigcap\mscr{C}$ and suppose for the sake of contradiction that there exists $g\in G$ such that $\mf{k}\subsetneq g\mf{k}$. Pick a point $x\in g\mf{k}\setminus\mf{k}$. Since $x\not\in\mf{k}$, there exists $\mf{h}\in\mscr{C}$ such that $x\not\in\mf{h}$. Since $\mf{h}\in\mc{H}_0(G)$, there exists $m\geq 1$ such that $g^m\mf{h}=\mf{h}$. Then $x\in g\mf{k}\cu g^m\mf{k}\cu g^m\mf{h}=\mf{h}$, a contradiction.
\end{rmk}

\begin{lem}\label{H_1 not saturated}
Suppose that $\mc{H}_{1/2}(G)=\emptyset$. Then, for every $\mf{h}\in\mc{H}_1(G)$ there exist $x\in\mf{h}$ and $y\in\mf{h}^*$ such that $\mscr{H}(x|y)\cu\mc{H}_1(G)$.
\end{lem}
\begin{proof}
Since $\mf{h}\in\mc{H}_1(G)$, there exists $g\in G$ such that $g\mf{h}\subsetneq\mf{h}$. Pick a point $x\in\mf{h}\setminus g\mf{h}$. Recall that $\s_x\cu\mscr{H}(M)$ is the set of halfspaces that contain $x$, and define $\s:=\s_x\cap\mc{H}_0(G)$. We will show that the intersection of all elements of the set $\s\sqcup\{\mf{h}^*\}$ is nonempty. Picking any point $y$ in this intersection will then yield $\mscr{H}(x|y)\cap\mc{H}_0(G)=\emptyset$, hence $\mscr{H}(x|y)\cu\mc{H}_1(G)$.

Let us show that $\mc{H}=\s\sqcup\{\mf{h}^*\}$ satisfies the hypotheses of Lemma~\ref{nonempty intersection criterion}. First, for any chain $\mscr{C}\cu\mc{H}$, either $\mf{h}^*$ is a lower bound for $\mscr{C}$, or the subset $\mscr{C}\cap\s$ is cofinal. In the latter case, the intersection of all elements of $\mscr{C}\cap\s$ is nonempty, since it contains $x$, so it must be a halfspace $\mf{k}$ by Remark~\ref{bigcap chain halfspace}. By Remark~\ref{miracle on H_0}, we have $\mf{k}\not\in\mc{H}_1(G)$. Therefore $\mf{k}\in\mc{H}_0(G)$, hence $\mf{k}\in\s\cu\mc{H}$. This shows that every chain in $\mc{H}$ has a lower bound in $\mc{H}$.

To conclude the proof, we need to show that any two elements of $\mc{H}$ intersect. Since all elements of $\s$ contain $x$, this boils down to showing that $\mf{h}^*$ cannot be disjoint from a halfspace $\mf{j}\in\s$. If this were the case, a finite-index subgroup $H\leq G$ would preserve $\mf{j}$. For every $h\in H$, we would have $ hx\in h\mf{j}=\mf{j}\cu\mf{h}$. However, since $g\mf{h}\subsetneq\mf{h}$ and $x\in\mf{h}\setminus g\mf{h}$, we have $g^{-n}x\in\mf{h}^*$ for all $n\geq 1$. For $n$ sufficiently large, $g^{-n}$ will lie in $H$, so we have obtained a contradiction.
\end{proof}

The next lemma will only be used in the proof of Proposition~\ref{gate-convex C bar} in Subsection~\ref{non-transverse sect}.

\begin{lem}\label{H_1 is not intersection of halfspaces of core}
Let $C\cu M$ be a $G$--invariant convex subset. If $\mscr{C}\cu\mscr{H}_C(M)$ is a chain such that $\bigcap\mscr{C}\in\mscr{H}(M)\setminus\mscr{H}_C(M)$, then $\bigcap\mscr{C}\not\in\mc{H}_1(G,M)$.
\end{lem}
\begin{proof}
Set $\mf{k}:=\bigcap\mscr{C}$. Since $\mf{k}^*$ contains the complement of any element of $\mscr{C}\cu\mscr{H}_C(M)$, we have $\mf{k}^*\cap C\neq\emptyset$. Thus, the fact that $\mf{k}\not\in\mscr{H}_C(M)$ implies that $\mf{k}\cap C=\emptyset$. 

Suppose for the sake of contradiction that there exists $g\in G$ such that $g\mf{k}\supsetneq\mf{k}$. Replacing $\mscr{C}$ with a cofinal subset, we can assume that some point $x_0\in g\mf{k}\setminus\mf{k}$ lies in $\mf{h}^*$ for all $\mf{h}\in\mscr{C}$.

\smallskip
{\bf Claim:} \emph{for every $\mf{h}\in\mscr{C}$ and $n\geq 1$, there exists a halfspace $\mf{j}_n(\mf{h})\in\mscr{C}$ such that, for every $\mf{j}\in\mscr{C}$ with $\mf{j}\cu\mf{j}_n(\mf{h})$, the halfspaces $\mf{h}$ and $g^n\mf{j}$ are transverse.}

\smallskip \noindent
\emph{Proof of Claim.} Since we can always replace $g$ with a proper power, it suffices to consider $n=1$.

To begin with, for every $\mf{j}\in\mscr{C}$, we have $\mf{h}\cap g\mf{j}\supseteq\mf{k}\cap g\mf{k}\neq\emptyset$ and $x_0\in\mf{h}^*\cap g\mf{j}$. Suppose for the sake of contradiction that the claim fails. Then there exists a cofinal subset $\mscr{C}'\cu\mscr{C}$ such that either $\mf{h}\cap g\mf{j}^*=\emptyset$ for all $\mf{j}\in\mscr{C}'$, or $\mf{h}^*\cap g\mf{j}^*=\emptyset$ for all $\mf{j}\in\mscr{C}'$ (equivalently, either $\mf{h}\cu g\mf{j}$ or $\mf{h}^*\cu g\mf{j}$). It follows that $g\mf{k}=\bigcap g\mscr{C}=\bigcap g\mscr{C}'$ contains either $\mf{h}$ or $\mf{h}^*$, which both lie in $\mscr{H}_C(M)$. This contradicts the fact that $g\mf{k}\cap C=g(\mf{k}\cap C)=\emptyset$. 
\hfill$\blacksquare$

\smallskip
Pick any halfspace $\mf{h}_0\in\mscr{C}$. Inductively choose $\mf{h}_{i+1}\in\mscr{C}$ with $\mf{h}_{i+1}\cu\mf{j}_{i+1}(\mf{h}_0)\cap\mf{j}_i(\mf{h}_1)\cap\dots\cap\mf{j}_1(\mf{h}_i)$; these halfspaces exist because the chain $\mscr{C}$ does not have a minimal element, as $\mf{k}\cap C=\emptyset$. By the Claim, the infinitely many halfspaces $g^i\mf{h}_i$ are pairwise transverse. This contradicts the fact that $M$ has finite rank, proving the lemma.
\end{proof}

The following two remarks are needed in the proof of Lemma~\ref{H_0 perp H_1} in Subsection~\ref{compatible metrics sect}.

\begin{rmk}\label{empty intersection implies transversality}
Consider $g\in\aut M$ and $\mf{h}\in\mc{H}_1(g)$ such that $\bigcap_{n\in\Z}g^n\mf{h}=\emptyset$ and $\bigcap_{n\in\Z}g^n\mf{h}^*=\emptyset$. Then $\mf{h}$ is transverse to every element of $\mc{H}_0(g)$. 

Indeed, suppose for the sake of contradiction that $\mf{k}\in\mc{H}_0(g)$ is not transverse to $\mf{h}$. Possibly replacing $\mf{k}$ and $\mf{h}$ with their complements, we can assume that $\mf{k}\cu\mf{h}$. Replacing $g$ with a nontrivial power, we can also assume that $g\mf{k}=\mf{k}$ and $g\mf{h}\subsetneq\mf{h}$. But then we have $\mf{k}=g^n\mf{k}\cu g^n\mf{h}$ for all $n\in\Z$, contradicting the fact that $\bigcap_{n\in\Z}g^n\mf{h}=\emptyset$.
\end{rmk}

The next remark is also needed in the proof of Proposition~\ref{cores of subalgebras} in Subsection~\ref{non-transverse sect}.

\begin{rmk}\label{dynamics in subalgebras}
Let $N\cu M$ be a $G$--invariant subalgebra. Consider $\mf{h}\in\mscr{H}(N)$. By Remark~\ref{halfspaces of subsets}, there exists $\mf{k}\in\mscr{H}(M)$ with $\mf{h}=\mf{k}\cap N$. If $\mf{h}\in\mc{H}_1(G,N)$, every such $\mf{k}$ lies in $\mc{H}_1(G,M)$.

Indeed, there exists $g\in G$ such that $g\mf{h}\subsetneq\mf{h}$. It is clear that $\mf{k}\not\in\mc{H}_0(g,M)$, or a power of $g$ would have to preserve $\mf{k}$ and $\mf{h}$. For every $n\in\Z$, the sets $g^n\mf{h}\cap\mf{h}$ and $g^n\mf{h}^*\cap\mf{h}^*$ are nonempty. Hence  $g^n\mf{k}\cap\mf{k}$ and $g^n\mf{k}^*\cap\mf{k}^*$ are nonempty, and we have $\mf{k}\not\in\overline{\mc{H}}_{1/2}(g,M)\sqcup\overline{\mc{H}}_{1/2}(g,M)^*$. We conclude that $\mf{k}\in\mc{H}_1(g,M)\cu\mc{H}_1(G,M)$.
\end{rmk}

\subsection{Convex cores.}\label{convex cores sect}

It is often possible to reduce to the case of actions satisfying $\mc{H}_{1/2}(G)=\emptyset$ or $\overline{\mc{H}}_{1/2}(G)=\emptyset$ by restricting to the following canonical subspaces.

\begin{defn}
\begin{enumerate}
\item[]
\item The \emph{(convex) core} $\mc{C}(G)$ is the intersection of all halfspaces in $\mc{H}_{1/2}(G)$. 
\item The \emph{reduced (convex) core} $\overline{\mc{C}}(G)$ is the intersection of all halfspaces in $\overline{\mc{H}}_{1/2}(G)$.
\end{enumerate}
We will write $\mc{C}(G,M)$ or $\overline{\mc{C}}(G,M)$ when it is necessary to specify the median algebra.
\end{defn}

The core and the reduced core are $G$--invariant convex subsets of $M$, although they can of course be empty in general. Note that we always have $\overline{\mc{C}}(G)\cu\mc{C}(G)$.

\begin{rmk}\label{core and commensurability}
If $H\leq G$ is a commensurated subgroup, then the core $\mc{C}(H)$ is $G$--invariant. 
Indeed, for every $g\in G$, Remark~\ref{G and h rmk} implies that:
\[\mc{H}_{\bullet}(H)=\mc{H}_{\bullet}(gHg^{-1}\cap H)=\mc{H}_{\bullet}(gHg^{-1})=g\mc{H}_{\bullet}(H).\]
Thus, the sets $\mc{H}_0(H)$, $\mc{H}_{1/2}(H)$ and $\mc{H}_1(H)$ are $G$--invariant, and so is $\mc{C}(H)$.

If $H$ is normal, then also the reduced core $\overline{\mc{C}}(H)$ and the sets $\overline{\mc{H}}_0(H)$, $\overline{\mc{H}}_{1/2}(H)$ are $G$--invariant.
\end{rmk}

The following is the main result of this subsection.

\begin{thm}\label{abstract core}
Let $G$ be finitely generated. Then: 
\begin{enumerate}
\item the core $\mc{C}(G)$ is nonempty;
\item if $G\acts M$ has no wall inversions, the reduced core $\overline{\mc{C}}(G)$ is nonempty.
\end{enumerate}
\end{thm}

\begin{ex}\label{empty reduced core}
If $G\acts M$ has wall inversions, the reduced core $\overline{\mc{C}}(G)$ can be empty. For instance, consider $M=[-1,1]\setminus\{0\}$ and the group $G$ generated by the reflection in $0$.
\end{ex}

We now obtain a sequence of lemmas leading up to Lemma~\ref{previously part of the proof}. Theorem~\ref{abstract core} will then immediately follow by applying Lemma~\ref{nonempty intersection criterion}.

\begin{lem}\label{basic group lemma}
For some $d\geq 1$, let $B_d$ be the (closed) ball of radius $d$ in a Cayley graph of $G$. Let $H\leq G$ be a subgroup with the property that $B_d\cu g_1H\cup\dots\cup g_dH$ for some $g_1,\dots,g_d\in G$. Then $H$ has index $\leq d$ in $G$.
\end{lem}
\begin{proof}
Let $S\cu G$ be the generating set corresponding to the chosen Cayley graph. We prove the lemma by induction on $d$. In the base case $d=1$, a coset of $H$ contains $\{1\}\cup S$, hence $H=G$.

Suppose that there exist $g_1,\dots,g_d\in G$ such that $B_d\cu g_1H\cup\dots\cup g_dH$. If $B_{d-1}$ is contained in the union of $d-1$ of these cosets, then the inductive hypothesis implies that $H$ has index $\leq d-1<d$. Thus, we can assume that $B_{d-1}$ intersects each of the cosets $g_1H,\dots,g_dH$.

Hence, for every $1\leq i\leq d$ and every $s\in S$, there exists $1\leq j(i,s)\leq d$ such that $sg_iH=g_{j(i,s)}H$. Since $S$ generates $G$, this implies that, for every $i$ and $g\in G$, there exists $1\leq j(i,g)\leq d$ such that $gg_iH=g_{j(i,g)}H$. Hence $G=g_1H\cup\dots\cup g_dH$ as required.
\end{proof}

\begin{lem}\label{efficient facing}
If $G$ is finitely generated, there exists a finite subset $F\cu G$ with the following property. For every $\mf{h}\in\overline{\mc{H}}_{1/2}(G)$, there exists $f\in F$ such that $f\mf{h}$ and $\mf{h}$ are facing.
\end{lem}
\begin{proof}
Choose any locally finite Cayley graph of $G$ and let $B\cu G$ be the (closed) ball of radius $r=\rk M$ in it. Take $F:=B^{-1}B$. Consider $\mf{h}\in\overline{\mc{H}}_{1/2}(G)$.

Let $G(\mf{h})\leq G$ denote the stabiliser of $\mf{h}$. Choose elements $g_1,\dots,g_k\in G$ such that the cosets $g_iG(\mf{h})$ are pairwise distinct. We take $k$ to be the index of $G(\mf{h})$ in $G$ if this is $\leq r$, and $k=r+1$ otherwise. Note that we can choose $g_1,\dots,g_k$ within $B$. Otherwise, $B$ would be covered by at most $k-1\leq r$ cosets of $G(\mf{h})$ and Lemma~\ref{basic group lemma} would imply that $G(\mf{h})$ has index $\leq k-1$ in $G$, a contradiction.

If $i\neq j$, Remark~\ref{1/2 bar rmk} guarantees that the halfspaces $g_i\mf{h}$ and $g_j\mf{h}$ are either transverse or facing. If $k\leq r$, then every coset of $G(\mf{h})$ is of the form $g_iG(\mf{h})$, and there exists $f\in\{g_1,\dots,g_k\}\cu B$ such that $f\mf{h}$ and $\mf{h}$ are facing. Otherwise, we have $k>r=\rk M$ and there exist $i\neq j$ such that $g_i\mf{h}$ and $g_j\mf{h}$ are facing. Setting $f=g_i^{-1}g_j$, we have $f\in B^{-1}B=F$, and $\mf{h}$ and $f\mf{h}$ are facing.
\end{proof}

\begin{lem}\label{previously part of the proof}
Let $G$ be finitely generated. 
\begin{enumerate}
\item If $\mscr{C}\cu\mc{H}_{1/2}(G)$ is a chain, then $\bigcap\mscr{C}$ is a halfspace of $M$ lying in $\mc{H}_{1/2}(G)$.
\item If $G\acts M$ has no wall inversions and $\mscr{C}\cu\overline{\mc{H}}_{1/2}(G)$ is a chain, then $\bigcap\mscr{C}\in\overline{\mc{H}}_{1/2}(G)$.
\end{enumerate}
\end{lem}
\begin{proof}
Recall that $\mc{H}_{1/2}(G)\cu\overline{\mc{H}}_{1/2}(G)$. The first part of the proof will deal with a general chain $\mscr{C}\cu\overline{\mc{H}}_{1/2}(G)$ and a general action $G\acts M$.  We will introduce the specific hypotheses of the two parts of the lemma only towards the end.

Set $C:=\bigcap\mscr{C}$. Given $g\in G$, we denote by $\mscr{C}(g)\cu\mscr{C}$ the subset of those halfspaces $\mf{h}$ such that $\mf{h}$ and $g\mf{h}$ are facing. We denote by $\mscr{S}(g)\cu\mscr{C}$ the subset of halfspaces $\mf{h}$ such that $g\mf{h}=\mf{h}$. 

By Lemma~\ref{efficient facing}, there exists a finite subset $F\cu G$ such that, for every $\mf{h}\in\overline{\mc{H}}_{1/2}(G)$, there exists $f\in F$ such that $\mf{h}$ and $f\mf{h}$ are facing. Thus, $\mscr{C}$ is the union of finitely many subsets $\mscr{C}(f)$ with $f\in F$, and there exists $f_0\in F$ such that $\mscr{C}(f_0)$ is cofinal in $\mscr{C}$.

Fix some halfspace $\mf{k}\in\mscr{C}(f_0)$. The set of those halfspaces $\mf{h}\in\mscr{C}(f_0)$ with $\mf{h}\cu\mf{k}$ is also cofinal in $\mscr{C}$. For every such $\mf{h}$, we have $\mf{k}^*\cu\mf{h}^*\cu f_0\mf{h}$, hence $f_0^{-1}\mf{k}^*\cu\mf{h}$. This shows that $f_0^{-1}\mf{k}^*\cu C$. In particular, $C$ is nonempty and Remark~\ref{bigcap chain halfspace} implies that $C\in\mscr{H}(M)$.

We now show that $C\not\in\mc{H}_1(G)$. Suppose for the sake of contradiction that $gC\subsetneq C$ for some $g\in G$. For every $\mf{h}\in\overline{\mc{H}}_{1/2}(G)$, the halfspaces $\mf{h},g\mf{h},\dots,g^r\mf{h}$ cannot be pairwise transverse, hence Remark~\ref{1/2 bar rmk} guarantees the existence of $1\leq i\leq r$ such that $\mf{h}$ and $g^i\mf{h}$ are either facing or equal. Thus, $\mscr{C}$ is the union of the finitely many subsets $\mscr{C}(g^i)\cup\mscr{S}(g^i)$, and there exists $1\leq i\leq r$ such that either $\mscr{C}(g^i)$ or $\mscr{S}(g^i)$ is cofinal in $\mscr{C}$. Hence either $g^iC^*\cap C^*=\emptyset$ or $g^iC=C$. However, $gC\subsetneq C$ implies that $g^iC\subsetneq C$ and $g^iC^*\cap C^*=C^*\neq\emptyset$, a contradiction. 

We conclude the proof separately for the two parts of the lemma.

\smallskip
{\bf Part~(2):} recalling that $\mscr{C}(f_0)$ is cofinal in $\mscr{C}$, we see that $f_0C^*\cap C^*=\emptyset$. Since $G\acts M$ has no wall inversions, this implies that $f_0C$ and $C$ are facing. We have already shown that $C\in\mscr{H}(M)\setminus\mc{H}_1(G)$, so this means that $C$ lies in $\overline{\mc{H}}_{1/2}(G)$, as required.

\smallskip
{\bf Part~(1):} as above, we have $f_0C^*\cap C^*=\emptyset$, hence $C\not\in\mc{H}_{1/2}(G)^*$. We are left to show that $C\not\in\mc{H}_0(G)$. If this were not the case, there would exist a finite-index subgroup $H\leq G$ such that $hC=C$ for all $h\in H$. By Remark~\ref{G and h rmk}, we would have $\mc{H}_{1/2}(H)=\mc{H}_{1/2}(G)\supseteq\mscr{C}$. Thus, applying Lemma~\ref{efficient facing} to $H$ and proceeding as above, we would obtain $h_0\in H$ such that $\mscr{C}(h_0)$ is cofinal in $\mscr{C}$. This would yield $h_0C^*\cap C^*=\emptyset$, contradicting the fact that $h_0C=C$.
\end{proof}

\begin{proof}[Proof of Theorem~\ref{abstract core}.]
Under the respective hypotheses, Remark~\ref{1/2 bar PWI} and Lemma~\ref{previously part of the proof} show that the sets $\mc{H}_{1/2}(G)$ and $\overline{\mc{H}}_{1/2}(G)$ satisfy the hypotheses of Lemma~\ref{nonempty intersection criterion}.
\end{proof}

We can now characterise the halfspaces of our two cores. Recall that, for every convex subset $C\cu M$, Remark~\ref{halfspaces of subsets} shows that the subset $\mscr{H}_C(M)\cu\mscr{H}(M)$ is naturally identified with $\mscr{H}(C)$ (and isomorphic to it as a pocset).

\begin{lem}\label{halfspaces of core}
\begin{enumerate}
\item[]
\item We have $\mscr{H}_{\mc{C}(G)}(M)\cu\mc{H}_0(G)\sqcup\mc{H}_1(G)$ and $\mscr{H}_{\overline{\mc{C}}(G)}(M)\cu\overline{\mc{H}}_0(G)\sqcup\mc{H}_1(G)$. Taking intersections with the cores, this induces $G$--invariant partitions:
\begin{align*}
\hspace{1.5cm} \mscr{H}(\mc{C}(G))&=\mc{H}_0(G,\mc{C}(G))\sqcup\mc{H}_1(G,\mc{C}(G)), & \mscr{H}(\overline{\mc{C}}(G))&=\overline{\mc{H}}_0(G,\overline{\mc{C}}(G))\sqcup\mc{H}_1(G,\overline{\mc{C}}(G)).
\end{align*}
\item If $G$ is finitely generated, then $\mc{H}_0(G)\cu\mscr{H}_{\mc{C}(G)}(M)$. If moreover $G\acts M$ has no wall inversions, then $\overline{\mc{H}}_0(G)\cu\mscr{H}_{\overline{\mc{C}}(G)}(M)$.
\end{enumerate}
\end{lem}
\begin{proof}
We begin with part~(1). The first sentence is clear. By Remark~\ref{halfspaces of subsets}, every $\mf{k}\in\mscr{H}(\mc{C}(G))$ is of the form $\mf{k}=\mf{h}\cap\mc{C}(G)$ for a unique $\mf{h}\in\mscr{H}_{\mc{C}(G)}(M)$ (and similarly for the reduced core).

If $\mf{h}\in\mc{H}_0(G,M)$, then $\mf{h}$ is preserved by a finite-index subgroup of $G$, and so is $\mf{h}\cap\mc{C}(G)$. In this case, we have $\mf{k}\in\mc{H}_0(G,\mc{C}(G))$. If $\mf{h}\in\overline{\mc{H}}_0(G,M)$, it follows from Remark~\ref{halfspaces of subsets} that $\mf{h}\cap\overline{\mc{C}}(G)$ lies in $\overline{\mc{H}}_0(G,\overline{\mc{C}}(M))$.

If $\mf{h}\in\mc{H}_1(G,M)$, then there exists $g\in G$ with $g\mf{h}\subsetneq\mf{h}$. In this case, $g\mf{k}\cu\mf{k}$ and, since $\mf{k}$ is the intersection with $\mc{C}(G)$ of a \emph{unique} halfspace of $M$, we must have $g\mf{k}\subsetneq\mf{k}$. Hence $\mf{k}\in\mc{H}_1(G,\mc{C}(G))$ and, similarly, $\mf{h}\cap\overline{\mc{C}}(G)\in\mc{H}_1(G,\overline{\mc{C}}(G))$.

Now, let us prove part~(2). Consider a halfspace $\mf{h}\in\mscr{H}(M)$. Applying Lemma~\ref{previously part of the proof} and Lemma~\ref{nonempty intersection criterion} to the sets $\mc{H}_{1/2}(G)\cup\{\mf{h}\}$ and $\mc{H}_{1/2}(G)\cup\{\mf{h}^*\}$, we see that $\mf{h}\not\in\mscr{H}_{\mc{C}(G)}(M)$ if and only if there exists $\mf{k}\in\mc{H}_{1/2}(G)$ such that $\mf{k}\cap\mf{h}=\emptyset$ or $\mf{k}\cap\mf{h}^*=\emptyset$. Similarly, if $G\acts M$ has no wall inversions, $\mf{h}\not\in\mscr{H}_{\overline{\mc{C}}(G)}(M)$ if and only if there exists $\mf{k}\in\overline{\mc{H}}_{1/2}(G)$ such that $\mf{k}\cap\mf{h}=\emptyset$ or $\mf{k}\cap\mf{h}^*=\emptyset$.

Halfspaces $\mf{k}\in\mc{H}_{1/2}(G)$ and $\mf{h}\in\mc{H}_0(G)$ can never be disjoint. Indeed, there exists a finite-index subgroup $H\leq G$ such that $g\mf{h}=\mf{h}$ for all $g\in H$. By Remark~\ref{G and h rmk}, we have $\mc{H}_{1/2}(G)=\mc{H}_{1/2}(H)$, so there exists $g\in H$ such that $g\mf{k}$ and $\mf{k}$ are facing. If $\mf{k}$ and $\mf{h}$ were disjoint, then $g\mf{h}^*$ and $\mf{h}^*$ would be facing, contradicting the fact that $g\mf{h}=\mf{h}$.

Finally, $\mf{k}\in\overline{\mc{H}}_{1/2}(G)$ and $\mf{h}\in\overline{\mc{H}}_0(G)$ can never be disjoint either. Indeed, there exists $g\in G$ such that $\mf{k}$ and $g\mf{k}$ are facing. If $\mf{k}$ and $\mf{h}$ were disjoint, then $g\mf{h}^*$ and $\mf{h}^*$ would be facing, violating the fact that, for every $g\in G$, either $g\mf{h}$ lies in $\{\mf{h},\mf{h}^*\}$ or it is transverse to $\mf{h}$.
\end{proof}

As a consequence of Theorem~\ref{abstract core}, we obtain the following generalisation of a classical result of Sageev on $\CAT$ cube complexes (see e.g.\ \cite[Proposition~B.8]{CFI} and \cite[Theorem~5.1]{Sag95}).

\begin{prop}\label{finite orbit new prop}
Suppose that $G$ is finitely generated and $\mc{H}_1(G)=\emptyset$. Then:
\begin{enumerate}
\item there exists a $G$--invariant median subalgebra $C\cu M$ isomorphic to $\{0,1\}^k$ with $0\leq k\leq r$;
\item if $G$ acts with no wall inversions, then $G$ fixes a point of $M$. 
\end{enumerate}
\end{prop}
\begin{proof}
By part~(1) of Theorem~\ref{abstract core} and part~(1) of Lemma~\ref{halfspaces of core}, we can pass to the core and assume that $\mscr{H}(M)=\mc{H}_0(G)$. Part~(1) then immediately follows from Lemmas~\ref{finite orbit new lem} and~\ref{from finite orbit to finite cube 1}.

Now, if we could deduce that the $G$--action on $C\simeq\{0,1\}^k$ has no wall inversions from the fact that $G\acts M$ has no wall inversions, then part~(2) would also follow immediately. Unfortunately, this is not true in general and part~(2) will require more work. Our objective will be to construct a \emph{specific} subalgebra $C\simeq\{0,1\}^k$ such that $G\acts C$ has no wall inversions.

We begin by passing to the reduced core and assuming that $\mscr{H}(M)=\overline{\mc{H}}_0(G)$. This is allowed because of part~(2) of Theorem~\ref{abstract core}, part~(1) of Lemma~\ref{halfspaces of core}, and Remark~\ref{gate-convex inversions}. By Remark~\ref{factors through finite group}, we can moreover assume that $G$ is a finite group.

Recall that every set of pairwise-transverse walls of $M$ has cardinality $\leq r=\rk M$. Thus, there exists a nonempty subset $\mc{U}\cu\mscr{W}(M)$ that is maximal among $G$--invariant sets of pairwise-transverse walls. Since $\mscr{H}(M)=\overline{\mc{H}}_0(G)$ and $G\acts M$ has no wall inversions, every $G$--orbit in $\mscr{W}(M)$ consists of pairwise-transverse walls. This implies that $\mc{U}$ is also maximal among (not necessarily $G$--invariant) sets of pairwise-transverse walls. Indeed, if there existed $\mf{w}\in\mscr{W}(M)$ with $\mc{U}\sqcup\{\mf{w}\}$ pairwise-transverse, then $\mc{U}\sqcup G\cdot\mf{w}$ would be pairwise-transverse and $G$--invariant.

In conclusion, there exists a maximal set of pairwise-transverse walls $\mc{U}\cu\mscr{W}(M)$ that is also $G$--invariant. Let $\mf{h}_1^+,\dots,\mf{h}_k^+$ be a choice of halfspace for every element of $\mc{U}$, and let $\mf{h}_1^-,\dots,\mf{h}_k^-$ be their complements. Choose a finite subset $F\cu M$ that intersects each of the $2^k$ sectors of the form $\mf{h}_1^{\pm}\cap\dots\cap\mf{h}_k^{\pm}$. Since $G$ is finite, we can take $F$ to be $G$--invariant.

Let $N\cu M$ be the median subalgebra generated by $F$. Note that $N$ is $G$--invariant and, by \cite[Lemma~4.2]{Bow-cm}, it is finite. It was observed in \cite[Section~3]{Bow4} that every finite median algebra is naturally isomorphic to the $0$--skeleton of a finite $\CAT$ cube complex. Let $X$ be the cube complex with $0$--skeleton isomorphic to $N$. Since the induced action $G\acts X^{(0)}$ preserves the median operator of $X$, it extends to an action by cubical automorphisms on the entire $X$. 

By our choice of $F$, the intersections $\mf{k}_i=\mf{h}_i^+\cap N$ are pairwise-transverse halfspaces of $N$. 
By part~(1) of Remark~\ref{halfspaces of subsets}, this is a \emph{maximal} collection of pairwise-transverse halfspaces of $N$. The corresponding walls of $N$ determine a maximal set of pairwise-transverse hyperplanes of $X$, which is also $G$--invariant. Thus, they correspond to a maximal cube $C\cu X$ that is preserved by $G$. 

We conclude the proof by showing that $G$ acts on $C$ without wall inversions (which implies that $G$ fixes a vertex of $C$). Suppose for the sake of contradiction that there exists an element $g\in G$ such that $g\mf{k}_i=\mf{k}_i^*$ for some $i$. This implies that $g\mf{h}_i^+\cap\mf{h}_i^+\cap N=\emptyset$. However, since $\mc{U}$ is $G$--invariant, there exists an index $j$ with $g\mf{h}_i^+=\mf{h}_j^{\pm}$. Thus, $g\mf{h}_i^+\cap\mf{h}_i^+\cap N$ contains one of the two nonempty sets $\mf{h}_j^{\pm}\cap\mf{h}_i^+\cap F$. This is the required contradiction.
\end{proof}

Recall from Definition~\ref{ess/min defn} that the action $G\acts M$ is said to be \emph{essential} if $\mscr{H}(M)=\mc{H}_1(G)$. One can often reduce to studying essential actions by relying on the following result, which also proves part~(2) of Theorem~\ref{cores intro}.

\begin{prop}\label{H=H_1}
If $G$ is finitely generated and $G\acts M$ has no wall inversions, then there exists a nonempty, $G$--invariant, convex subset $C\cu M$ such that $\mscr{H}_C(M)\cu\mc{H}_1(G,M)$.
\end{prop}
\begin{proof}
By Theorem~\ref{abstract core}, the core $\mc{C}(G)\cu M$ is nonempty. By Remark~\ref{gate-convex inversions}, the action $G\acts\mc{C}(G)$ has no wall inversions. Thus, it is not restrictive to assume that $M=\mc{C}(G)$, i.e.\ that $\mc{H}_{1/2}(G,M)=\emptyset$. 

Let $\mc{U}\cu\mscr{W}(M)$ be the set of walls that bound halfspaces in $\mc{H}_0(G,M)$. Let $\pi_{\mc{U}}\colon M\ra M(\mc{U})$ be the corresponding restriction quotient, as defined in Subsection~\ref{restriction quotients sect}. The image of the induced map $\pi_{\mc{U}}^*\colon\mscr{H}(M(\mc{U}))\ra\mscr{H}(M)$ coincides with the set of $\sim_{\mc{U}}$--saturated halfspaces of $M$. By Lemma~\ref{H_1 not saturated}, no element of $\mc{H}_1(G,M)$ is $\sim_{\mc{U}}$--saturated. Hence the image of $\pi_{\mc{U}}^*$ is exactly $\mc{H}_0(G,M)$. 

We conclude that $\mc{H}_0(G,M(\mc{U}))=\mscr{H}(M(\mc{U}))$. It is clear that $G\acts M(\mc{U})$ has no wall inversions, so part~(2) of Proposition~\ref{finite orbit new prop} yields a $G$--fixed point $x_0\in M(\mc{U})$. The fibre $C:=\pi_{\mc{U}}^{-1}(x_0)$ is nonempty, $G$--invariant and convex. Since $C$ consists of a single $\sim_{\mc{U}}$--equivalence class, $\mscr{H}_C(M)$ and $\mc{H}_0(G,M)$ are disjoint. Hence $\mscr{H}_C(M)\cu\mc{H}_1(G,M)$. 
\end{proof}

\begin{rmk}\label{H=H_1 rmk}
Even if $G\acts M$ does have wall inversions, the conclusion of Proposition~\ref{H=H_1} always holds for a subgroup $H\leq G$ of index at most $2^r$. This can be proved exactly as above, simply appealing to part~(1) of Proposition~\ref{finite orbit new prop} rather than part~(2).
\end{rmk}

The proof of Theorem~\ref{cores intro} is completed by the following lemma, which does not make any finite generation assumption on the group $G$. Recall from Definition~\ref{ess/min defn} that the action $G\acts M$ is \emph{minimal} if $M$ does not contain any proper, $G$--invariant, convex subsets (a convex subset is \emph{proper} if it is nonempty and not the entire $M$). 

\begin{lem}\label{essential vs minimal cor}
\begin{enumerate}
\item[]
\item If $G\acts M$ is essential, then it is minimal.
\item If $G\acts M$ is minimal and without wall inversions, then it is essential.
\end{enumerate}
\end{lem}
\begin{proof}
We begin with the proof of part~(1). Suppose for the sake of contradiction that $G\acts M$ is essential, but there exists a proper, $G$--invariant, convex subset $C\cu M$. Let $\s_C\cu\mscr{H}(M)$ be the set of halfspaces containing $C$. Note that $\s_C$ is $G$--invariant and, since $C\neq M$, we have $\s_C\neq\emptyset$. 

Every chain $\mscr{C}\cu\s_C$ has a lower bound in $C$. Indeed, $\bigcap\mscr{C}$ is nonempty, so it must be a halfspace of $M$ containing $C$. Zorn's lemma guarantees the existence of a minimal element $\mf{h}\in\s_C$. However, since $G\acts M$ is essential, there exists $g\in G$ such that $g\mf{h}\subsetneq\mf{h}$. Since $\s_C$ is $G$--invariant, we have $g\mf{h}\in\s_C$, obtaining the required contradiction.

We now prove part~(2). Let $G\acts M$ be minimal and without wall inversions. If this action were not essential, there would exist a halfspace $\mf{h}\in\overline{\mc{H}}_0(G)\cup\overline{\mc{H}}_{1/2}(G)$. Remark~\ref{1/2 bar PWI} and the fact that there are no wall inversions imply that the orbit $G\cdot\mf{h}$ consists of pairwise intersecting halfspaces. In addition, since $\mf{h}\not\in\mc{H}_1(G)$, the only chains contained in $G\cdot\mf{h}$ are singletons. Thus, Lemma~\ref{nonempty intersection criterion} implies that $\bigcap(G\cdot\mf{h})$ is a proper, $G$--invariant, convex subset of $M$, violating minimality of the $G$--action.
\end{proof}

\begin{rmk}\label{essential vs minimal rmk}
When wall inversions are allowed, one can still conclude that $G\acts M$ is essential if the restriction $H\acts M$ is minimal for every subgroup $H\leq G$ of index $\leq 2^r$.

Indeed, by the same argument as in Remark~\ref{1/2 bar PWI}, if $\mf{h},\mf{k}\in\mscr{H}(M)\setminus\mc{H}_1(G)$ are halfspaces with $\mf{h}^*\in G\cdot\mf{h}$ and $\mf{k}^*\in G\cdot\mf{k}$, then $\mf{h}$ and $\mf{k}$ are transverse. Thus, there are at most $r$ such halfspaces and a subgroup $H\leq G$ of index $\leq 2^r$ leaves each of them invariant. The subgroup $H$ can only invert halfspaces in $\mc{H}_1(G)=\mc{H}_1(H)$, so we can run the argument used to prove part~(2) of Lemma~\ref{essential vs minimal cor}.
\end{rmk}

Proposition~\ref{H=H_1} and Lemma~\ref{essential vs minimal cor} prove Theorem~\ref{cores intro} from the introduction.

\subsection{Semisimple automorphisms.}\label{ss sect}

This subsection is devoted to deducing from Proposition~\ref{H=H_1} that automorphisms of finite-rank median algebras are (almost) semisimple.

The key result here is Lemma~\ref{ss lemma}, which makes significant use of the notion of \emph{zero-completion} of a median algebra. We briefly recall all necessary information on this object, which was originally introduced by Bandelt and Meletiou in \cite{Bandelt-Meletiou} and which we further developed in \cite[Subsection~4.1]{Fio1}. The reader should keep in mind the situation where $M$ is the vertex set of a $\CAT$ cube complex, in which case the zero-completion is just the Roller compactification.

In general, the zero-completion of $M$ is the median algebra $\overline M$ defined as follows. Consider the set of all intervals $\I:=\{I(x,y)\mid x,y\in M\}$. Points of $\overline M$ are functions $x\colon\I\ra M$ satisfying:
\begin{itemize}
\item $x(I)\in I$ for all $I\in\I$;
\item $\pi_{I\cap J}(x(I))=\pi_{I\cap J}(x(J))$ for all $I,J\in\I$ with $I\cap J\neq\emptyset$, where $\pi_{I\cap J}\colon M\ra I\cap J$ denotes the gate-projection.
\end{itemize}
The median operator of $\overline M$ is given by $m(x,y,z)(I)=m(x(I),y(I),z(I))$. We can define a map $\iota\colon M\ra\overline M$ by setting $\iota(x)(I)=\pi_I(x)$, where $\pi_I\colon M\ra I$ is the gate-projection. 

In \cite[Subsection~4.1]{Fio1}, we showed the following:

\begin{thm}
The zero-completion $\overline M$ is a median algebra with $\rk\overline M=\rk M$. The map $\iota\colon M\ra\overline M$ is an $(\aut M)$--equivariant, injective median morphism with convex image.
\end{thm}

Thus, we will not distinguish between $M$ and $\iota(M)\cu\overline M$. Moreover, we will identify the sets $\mscr{H}_M(\overline M)$ and $\mscr{H}(M)$, which is allowed by part~(2) of Remark~\ref{halfspaces of subsets}. We write $\partial M:=\overline M\setminus M$.

Note that a point $(x\colon\I\ra M)\in\overline M$ lies in a halfspace $\mf{h}\in\mscr{H}(M)\simeq\mscr{H}_M(\overline M)$ if and only if $x(I)\in\mf{h}$ for every $I$ such that $\mf{h}\in\mscr{H}_I(M)$ (see \cite[pp.~1356--1357]{Fio1}).

We are now ready to state the main result of this subsection.

\begin{lem}\label{ss lemma}
Consider $g\in\aut M$. Suppose that $\mc{H}_1(g)=\mscr{H}(M)$. Denote by $\mc{H}^+$ (resp. $\mc{H}^-$) the set of halfspaces $\mf{h}\in\mscr{H}(M)$ such that there exists $n\geq 1$ with $g^n\mf{h}\subsetneq\mf{h}$ (resp.\ $\mf{h}\subsetneq g^n\mf{h}$) .
\begin{enumerate}
\item For every $\mf{h}\in\mscr{H}(M)$, we have $\bigcap_{n\in\Z}g^n\mf{h}=\emptyset$ and $\bigcup_{n\in\Z}g^n\mf{h}=M$.
\end{enumerate} 
\begin{enumerate}
\setcounter{enumi}{1}
\item We have a $\langle g\rangle$--invariant partition $\mscr{H}(M)=\mc{H}^-\sqcup\mc{H}^+$. Any two halfspaces in $\mc{H}^+$ intersect, and so do any two halfspaces in $\mc{H}^-$. Moreover, $\mc{H}^-=(\mc{H}^+)^*$.
\item There exists a unique point $\xi^+\in\partial M$ (resp.\ $\xi^-\in\partial M$) that lies in every halfspace in $\mc{H}^+$ (resp.\ $\mc{H}^-$). We have $g\xi^+=\xi^+$, $g\xi^-=\xi^-$ and $M\cu I(\xi^-,\xi^+)$.
\item There exists $x\in M$ such that the sets $\mscr{W}(g^nx|g^{n+1}x)$ are pairwise disjoint for $n\in\Z$.
\item If no $\mf{h}\in\mscr{H}(M)$ is transverse to $g\mf{h}$, then part~(4) holds for all $x\in M$.
\end{enumerate}
\end{lem}
\begin{proof}
We begin with part~(1). For every $\mf{h}\in\mscr{H}(M)=\mc{H}_1(g)$, there exists $k\in\Z$ such that $g^k\mf{h}\subsetneq\mf{h}$. Note that $\bigcap_{n\in\Z}g^{nk}\mf{h}$ and its complement $\bigcup_{n\in\Z}g^{nk}\mf{h}^*$ are convex and $\langle g^k\rangle$--invariant. 
Thus, if $\bigcap_{n\in\Z}g^{nk}\mf{h}$ were nonempty, it would be a $\langle g^k\rangle$--invariant halfspace, contradicting the assumption that $\mc{H}_0(g)=\emptyset$. We conclude that $\bigcap_{n\in\Z}g^n\mf{h}\cu\bigcap_{n\in\Z}g^{nk}\mf{h}=\emptyset$. Since this holds for every $\mf{h}\in\mscr{H}(M)$, we also have $\bigcap_{n\in\Z}g^n\mf{h}^*=\emptyset$ and hence $\bigcup_{n\in\Z}g^n\mf{h}=M$. 

Let us prove part~(2). The only statement requiring a proof is that any two halfspaces $\mf{h},\mf{k}\in\mc{H}^+$ must intersect. The fact that $\mc{H}^-$ has the same property can be proved similarly. 

Suppose for the sake of contradiction that $\mf{h}\cap\mf{k}=\emptyset$ and that there exist $k,m\geq 1$ with $g^k\mf{h}\subsetneq\mf{h}$ and $g^m\mf{k}\subsetneq\mf{k}$. Replacing $k$ and $m$ by $km$, we can assume that $k=m$. Then, for every $n\geq 1$, we have $g^{kn}\mf{k}\subsetneq\mf{k}\cu\mf{h}^*$, hence $\mf{k}\cu g^{-kn}\mf{h}^*\subsetneq\mf{h}^*$. Thus, $\emptyset\neq\mf{k}\cu\bigcap_{n\in\Z}g^{kn}\mf{h}^*$, which contradicts part~(1).

We now prove part~(3). Let $I\cu M$ be an interval and let $\s_I\cu\mscr{H}(M)$ be the set of halfspaces that contain it. By Lemma~\ref{nonempty intersection criterion}, the intersection of all halfspaces in the set $\s_I\sqcup(\mc{H}^+\cap\mscr{H}_I(M))$ is nonempty. Since $\mc{H}^+$ contains a side of every wall of $M$, this intersection consists of a single point, which we denote by $z_I^+$. Note that $z_I^+\in I$. Similarly, there exists a unique point $z_I^-\in I$ that lies in every halfspace of the set $\s_I\sqcup(\mc{H}^-\cap\mscr{H}_I(M))$. It is clear that every wall of $I$ separates the points $z_I^{\pm}$, hence $I=I(z_I^-,z_I^+)$. It is straightforward to check that the functions $I\mapsto z_I^+$ and $I\mapsto z_I^-$ define points $\xi^+,\xi^-\in\overline M$ such that $M\cu I(\xi^-,\xi^+)$, and that these points are fixed by $g$. Since $\mc{H}_1(g)=\mscr{H}(M)$, no point of $M$ is fixed by $g$, so we must have $\xi^+,\xi^-\in\partial M$, completing the proof of part~(3).

Let us address part~(4). We begin with the following:

\smallskip
{\bf Claim:} \emph{for every $x\in M$, the set $\mscr{W}(x,\xi^+ \mid gx,g^2x,\dots,g^rx)$ is empty.} 

\smallskip\noindent
\emph{Proof of Claim.} Consider a halfspace $\mf{h}\in\mscr{H}(x,\xi^+ \mid gx,g^2x,\dots,g^rx)$. Recall from part~(3) that $M\cu I(\xi^-,\xi^+)$. Thus, since $\xi^+\in\mf{h}^*$, we must have $\xi^-\in\mf{h}$. 

For each $1\leq i\leq r$, either $g^i\mf{h}\subsetneq\mf{h}$ or the halfspaces $g^i\mf{h}$ and $\mf{h}$ are transverse. Indeed, observing that $\mf{h}\in\mscr{H}(\xi^+ \mid g^ix,\xi^-)$ and $g^i\mf{h}\in\mscr{H}(g^ix,\xi^+ \mid \xi^-)$, we see that the three intersections $\mf{h}\cap g^i\mf{h}$, $\mf{h}\cap g^i\mf{h}^*$ and $\mf{h}^*\cap g^i\mf{h}^*$ are nonempty, as they contain the points $\xi^-$, $g^ix$ and $\xi^+$, respectively. 

Since $\mf{h},g\mf{h},\dots,g^r\mf{h}$ cannot be pairwise transverse, there exists $1\leq k\leq r$ such that $g^k\mf{h}\subsetneq\mf{h}$. However, this implies that $\mf{h}\in\mc{H}^+$, contradicting the fact that $\xi^+\in\mf{h}^*$. 
\hfill$\blacksquare$

\smallskip
Now, pick any point $x_0\in M$ and define iteratively $x_{i+1}=m(x_i,gx_i,\xi^+)$ for $i\geq 0$. Using repeatedly the identities from Lemma~\ref{lem:identities} (applied to the zero-completion $\overline M$), we obtain the following chain of equalities:
\begin{align*}
\mscr{W}(gx_i \mid x_{i+1})&=\mscr{W}(gx_i \mid m(x_i,gx_i,\xi^+))=\mscr{W}(gx_i \mid \xi^+,x_i) \\
&=\mscr{W}\left(m(gx_{i-1},g^2x_{i-1},\xi^+) \mid \xi^+,m(x_{i-1},gx_{i-1},\xi^+)\right) \\
&=\mscr{W}\left(gx_{i-1},g^2x_{i-1} \mid \xi^+,x_{i-1}\right) \\
&=\mscr{W}\left(m(gx_{i-2},g^2x_{i-2},\xi^+),m(g^2x_{i-2},g^3x_{i-2},\xi^+) \mid \xi^+,m(x_{i-2},gx_{i-2},\xi^+)\right) \\
&=\mscr{W}\left(gx_{i-2},g^2x_{i-2},g^3x_{i-2} \mid \xi^+,x_{i-2}\right) \\
&=\dots=\mscr{W}(gx_0,\dots,g^{i+1}x_0 \mid \xi^+,x_0).
\end{align*} 
For $i\geq r-1$, the Claim implies that $\mscr{W}(gx_i \mid x_{i+1})=\emptyset$, hence $gx_i=x_{i+1}\in I(x_i,\xi^+)$. 

It follows that the sets $\mscr{W}(g^nx_{r-1} \mid g^{n+1}x_{r-1})=\mscr{W}(x_{n+r-1} \mid x_{n+r})$ are pairwise disjoint for $n\geq 0$. Applying $g^{-k}$ to these sets, with $k\geq 1$, we deduce that the sets $\mscr{W}(g^nx_{r-1} \mid g^{n+1}x_{r-1})$ are actually pairwise disjoint for all $n\in\Z$. This proves part~(4). 

Finally, observe that, under the hypothesis of part~(5), we actually have $\mscr{W}(gx_0 \mid x_1)=\emptyset$ for every $x_0\in M$. In order to see this, suppose for the sake of contradiction that there exists a halfspace $\mf{h}\in\mscr{H}(x_1 \mid gx_0)=\mscr{H}(x_0,\xi^+ \mid gx_0)$. As in the proof of the Claim, this implies that either $g\mf{h}\subsetneq\mf{h}$ or $g\mf{h}$ and $\mf{h}$ are transverse. The former is ruled out by the fact that $\xi^+\in\mf{h}^*$, whereas the latter would contradict the hypothesis of part~(5). 

As in part~(4), the fact that $\mscr{W}(gx_0 \mid x_1)$ is empty implies that the sets $\mscr{W}(g^nx_0 \mid g^{n+1}x_0)$ are pairwise disjoint for $n\in\Z$. This concludes the proof of part~(5) and the entire lemma.
\end{proof}

Recall from Definition~\ref{minset defn} that $g\in\aut M$ is \emph{semisimple} if there exists $x\in M$ such that the sets $\mscr{W}(g^nx|g^{n+1}x)$ are pairwise disjoint for $n\in\Z$. We say that $g$ acts \emph{stably without wall inversions} if the action $\langle g\rangle\acts M$ has no wall inversions.

\begin{cor}\label{semisimple cor general}
\begin{enumerate}
\item[]
\item If $g\in\aut M$ acts stably without wall inversions, then $g$ is semisimple.
\item For every $g\in\aut M$, there exists $1\leq i\leq 2^r$ such that $g^i$ is semisimple.
\end{enumerate}
\end{cor}
\begin{proof}
This follows from part~(4) of Lemma~\ref{ss lemma}, along with Proposition~\ref{H=H_1} for part~(1) and Remark~\ref{H=H_1 rmk} for part~(2).
\end{proof}

Corollary~\ref{semisimple cor general} implies Corollary~\ref{ss cor intro} from the introduction. This essentially also yields Corollary~\ref{main median semisimple}, although we will only provide a complete proof towards the end of Section~\ref{compatible metrics sect}.

\subsection{Non-transverse automorphisms.}\label{non-transverse sect}

In this subsection, we consider automorphisms of $M$ satisfying the following property. 

\begin{defn}\label{non-transverse defn}
An element $g\in\aut M$ acts \emph{non-transversely} if there does not exist $\mf{w}\in\mscr{W}(M)$ such that $\mf{w}$ and $g\mf{w}$ are transverse. An action $G\acts M$ is \emph{non-transverse} if every $g\in G$ acts non-transversely\footnote{A priori, requiring $g$ to act non-transversely is weaker than asking that the action $\langle g\rangle\acts M$ be non-transverse.}.
\end{defn}

The following is the main motivating example for Definition~\ref{non-transverse defn}.

\begin{ex}
Consider actions by median automorphisms $G\acts T_i$, where each $T_i$ is a median algebra of rank $1$. For instance, these could correspond to isometric $G$--actions on $\R$--trees. Consider the diagonal $G$--action $G\acts T_1\x\dots\x T_k$. If there exists a $G$--equivariant, injective median morphism $M\hookrightarrow T_1\x\dots\x T_k$, then the action $G\acts M$ is non-transverse by part~(1) of Remark~\ref{halfspaces of subsets}.
\end{ex}

Let us make two simple observations that will be useful later in this subsection.

\begin{rmk}\label{non-transverse rmk}
If $g$ acts non-transversely, then, for every $\mf{h}\in\overline{\mc{H}}_{1/2}(g)$, the halfspaces $\mf{h}$ and $g\mf{h}$ are facing, and so are $\mf{h}$ and $g^{-1}\mf{h}$. If $g$ acts non-transversely and stably without inversions, then $\mf{h}\in\overline{\mc{H}}_0(g)$ if and only if $g\mf{h}=\mf{h}$. 
\end{rmk}

\begin{rmk}\label{few possibilities for H_1(g)}
Let $g\in\aut M$ act non-transversely and stably without inversions. Consider a halfspace $\mf{h}\in\mc{H}_1(g)\cap\mscr{H}_{\overline{\mc{C}}(g)}(M)$. Then either $g\mf{h}\subsetneq\mf{h}$ or $g\mf{h}\supsetneq\mf{h}$. 

Indeed, since $\mf{h}\in\mc{H}_1(g)$, there exists $m\in\Z$ such that $g^m\mf{h}\subsetneq\mf{h}$. Helly's lemma guarantees the existence of a point $x\in\mf{h}\cap g^m\mf{h}^*\cap\overline{\mc{C}}(g)$. Let $C_x\cu\overline{\mc{C}}(g)$ be the convex hull of $\langle g\rangle\cdot x$. By Remark~\ref{non-transverse rmk}, $g$ preserves every halfspace in $\overline{\mc{H}}_0(g)$, which implies that $\mscr{H}_{C_x}(M)\cap\overline{\mc{H}}_0(g)=\mscr{H}_{\langle g\rangle\cdot x}(M)\cap\overline{\mc{H}}_0(g)=\emptyset$. Since $C_x\cu\overline{C}(g)$, it follows that $\mscr{H}_{C_x}(M)\cu\mc{H}_1(g)$, hence $\mscr{H}(C_x)=\mc{H}_1(g,C_x)$. 

Now, Lemma~\ref{ss lemma} shows that, for every $\mf{k}\in\mscr{H}(C_x)$, the intersections $g\mf{k}\cap\mf{k}$ and $g\mf{k}^*\cap\mf{k}^*$ are both nonempty. Since $x\in\mf{h}\cap g^m\mf{h}^*$, we have $\mf{h}\cap C_x\in\mscr{H}(C_x)$, hence $g\mf{h}\cap\mf{h}$ and $g\mf{h}^*\cap\mf{h}^*$ are nonempty. It follows that either $g\mf{h}\subsetneq\mf{h}$ or $g\mf{h}\supsetneq\mf{h}$ or $g\mf{h}$ and $\mf{h}$ are transverse. The latter is ruled out by the fact that $g$ acts non-transversely. 
\end{rmk}

Recall from Definition~\ref{minset defn} that the \emph{minimal set} $\Min(g)$ is the set of points $x\in M$ such that the sets $\mscr{W}(g^nx|g^{n+1}x)$ are pairwise disjoint for $n\in\Z$. When it is necessary to specify the median algebra under consideration, we will also write $\Min(g,M)$.

For non-transverse automorphisms, $\Min(g)$ turns out to be a \emph{convex} subset of $M$:

\begin{prop}\label{non-transverse prop}
Let $g\in\aut M$ act non-transversely and stably without inversions. Then $\Min(g)=\overline{\mc{C}}(g)$. Moreover, for every $x\in M$, we have $I(x,gx)\cap\overline{\mc{C}}(g)\neq\emptyset$.
\end{prop}
\begin{proof}
If $x\not\in\overline{\mc{C}}(g)$, then there exists $\mf{h}\in\overline{\mc{H}}_{1/2}(g)$ with $x\in\mf{h}^*$. By Remark~\ref{non-transverse rmk}, we have $gx\in g\mf{h}^*\cu\mf{h}$ and $g^{-1}x\in g^{-1}\mf{h}^*\cu\mf{h}$. In particular, $\mscr{W}(g^{-1}x|x)\cap\mscr{W}(x|gx)\neq\emptyset$, hence $x\not\in\Min(g)$. This shows the inclusion $\Min(g)\cu\overline{\mc{C}}(g)$. 

Let us prove that $I(x,gx)\cap\overline{\mc{C}}(g)\neq\emptyset$ for all $x\in M$. Given $x\in M$, consider the set of halfspaces $\mc{K}=\overline{\mc{H}}_{1/2}(g)\cup(\s_x\cap\s_{gx})$. The argument in the above paragraph shows $\mscr{H}(x,gx|\overline{\mc{C}}(g))\cap\overline{\mc{H}}_{1/2}(g)=\emptyset$. This implies that the halfspaces in $\mc{K}$ intersect pairwise. Any chain in $\mc{K}$ contains a cofinal subset contained in either $\overline{\mc{H}}_{1/2}(g)$ or $\s_x\cap\s_{gx}$. By Lemma~\ref{previously part of the proof}, any chain in $\overline{\mc{H}}_{1/2}(g)$ admits a lower bound in $\overline{\mc{H}}_{1/2}(g)$. It is clear that the set $\s_x\cap\s_{gx}$ has the same property. Thus, $\mc{K}$ satisfies the hypotheses of Lemma~\ref{nonempty intersection criterion}, which shows that $I(x,gx)\cap\overline{\mc{C}}(g)\neq\emptyset$.

We are only left to prove that $\overline{\mc{C}}(g)\cu\Min(g)$. Recall that, by Lemma~\ref{halfspaces of core}, we have a partition:
\[\mscr{H}(\overline{\mc{C}}(g))=\overline{\mc{H}}_0(g,\overline{\mc{C}}(g))\sqcup\mc{H}_1(g,\overline{\mc{C}}(g)).\]
By Remarks~\ref{halfspaces of subsets} and~\ref{gate-convex inversions}, $g$ also acts on $\overline{\mc{C}}(g)$ non-transversely and stably without inversions. By Remark~\ref{non-transverse rmk}, every element of $\overline{\mc{H}}_0(g,\overline{\mc{C}}(g))$ is preserved by $g$. If $C$ is the convex hull of any $\langle g\rangle$--orbit in $\overline{\mc{C}}(g)$, it follows that $\mscr{H}_C(M)\cu\mc{H}_1(g,M)$, hence $\mscr{H}(C)=\mc{H}_1(g,C)$. 

Thus, part~(5) of Lemma~\ref{ss lemma} shows that $\Min(g,C)=C$. We deduce that $\Min(g,\overline{\mc{C}}(g))=\overline{\mc{C}}(g)$ and $\overline{\mc{C}}(g)\cu\Min(g,M)$, completing the proof.
\end{proof}

\begin{prop}\label{gate-convex C bar}
If $g\in\aut M$ acts non-transversely and stably without inversions, then $\overline{\mc{C}}(g)$ is gate-convex. 
\end{prop}
\begin{proof}
Consider a chain $\mscr{C}\cu\mscr{H}_{\overline{\mc{C}}(g)}(M)$ such that $\bigcap\mscr{C}\neq\emptyset$. Then $\mf{k}:=\bigcap\mscr{C}\in\mscr{H}(M)$. We only need to show that $\mf{k}$ must intersect $\overline{\mc{C}}(g)$, since then we can invoke Lemma~\ref{gate-convexity criterion}.

By part~(1) of Lemma~\ref{halfspaces of core}, we can replace $\mscr{C}$ with a cofinal subset and assume that either $\mscr{C}\cu\mc{H}_1(g,M)$ or $\mscr{C}\cu\overline{\mc{H}}_0(g,M)$. In the latter case, $g$ preserves every halfspace in $\mscr{C}$, by Remark~\ref{non-transverse rmk}. Thus, $g\mf{k}=\mf{k}$, which implies that $\mf{k}\in\overline{\mc{H}}_0(g)$. By part~(2) of Lemma~\ref{halfspaces of core}, $\mf{k}$ intersects $\overline{\mc{C}}(g)$.

Let us suppose instead that $\mscr{C}\cu\mc{H}_1(g,M)$. By Remark~\ref{few possibilities for H_1(g)}, we can pass to a cofinal subset of $\mscr{C}$ and assume that $g\mf{h}\subsetneq\mf{h}$ for all $\mf{h}\in\mscr{C}$ (possibly also replacing $g$ with $g^{-1}$). Thus, $g\mf{k}\cu\mf{k}$. If $g\mf{k}=\mf{k}$, we conclude that $\mf{k}$ intersects $\overline{\mc{C}}(g)$ as above. Otherwise $g\mf{k}\subsetneq\mf{k}$ and Lemma~\ref{H_1 is not intersection of halfspaces of core} implies that $\mf{k}\in\mscr{H}_{\overline{\mc{C}}(g)}(M)$.
\end{proof}

The next three examples show, respectively, that Proposition~\ref{gate-convex C bar} does not extend to actions of general finitely generated groups, does not extend to the core $\mc{C}(g)$, and can fail if $g$ does not act non-transversely. 

\begin{ex}\label{non-gate-convex core in trees}
There are isometric actions of finitely generated free groups $F_n$ on complete $\R$--trees $T$ for which the reduced core $\overline{\mc{C}}(F_n,T)$ is not closed (hence not gate-convex).

This is because, for every isometric $G$--action on an $\R$--tree, the reduced core coincides with the minimal $G$--invariant sub-tree. As in \cite[Example~II.6]{Gaboriau-Levitt}, the latter needs not be complete, hence it can be non-closed. We recall the argument for the reader's convenience.

Let $G\acts T$ be an action on a complete $\R$--tree with dense orbits (and no global fixed point). For instance, several examples with $G=F_3$ were constructed in \cite[Theorem~5]{Levitt-Duke}. By \cite[Proposition~3.1]{Culler-Morgan}, the minimal invariant sub-tree $T_0\cu T$ is a union of axes of elements of $G$. Since all orbits are dense, branch points accumulate on every point of $T_0$. Thus, every geodesic is nowhere-dense in $T_0$. It follows that $T_0$ is a countable union of nowhere-dense subsets. This would violate Baire's theorem if $T_0$ were complete.

Note that isometric $G$--actions on $\R$--trees never have wall inversions, see e.g.\ Remark~\ref{connected inversions} below.
\end{ex}

\begin{ex}
Let $T$ be the rooted simplicial tree with root $v$ of degree $2$ and every other vertex of degree $3$. Orient every edge of $T$ away from $v$. Every vertex $w\in T$ is contained in two oriented edges moving away from $w$; we label them by $0$ and $1$, respectively. This gives a coding of the vertices of $T$ in terms of finite sequences of $0$s and $1$s, with $v$ corresponding to the empty sequence. 

Let $S_n\cu T$ be the set of $2^n$ vertices represented by sequences of length $n\geq 0$. If an edge $e\cu T$ intersects $S_{n-1}$ and $S_n$, we assign $e$ a length of $2^{-n}$. The metric completion $\wh T$ of the resulting metric on $T$ is a complete $\R$--tree. Note that $\wh T$ coincides with the closed ball of radius $1$ around $v$, and $T\cu\wh T$ is the corresponding open ball.

Points of $\wh T\setminus T$ are represented by infinite sequences of $0$s and $1$s. Thus, it is natural to identify the set $\wh T\setminus T$ with the dyadic integers $\Z_2$, i.e.\ with the set of expansions of the form $\sum_{i\geq 0}a_i2^i$ with $a_i\in\{0,1\}$. We then have an identification between $S_n$ and $\Z_2/2^n\Z_2$, i.e.\ with the set of expansions of the form $\sum_{0\leq i<n}a_i2^i$ with $a_i\in\{0,1\}$.

Let $g$ be the isometry of $\wh T$ that, under the identification between $\wh T\setminus T$ and $\Z_2$, is represented by the addition of $+1$ in the ring $\Z_2$. Then $g$ fixes the root $v$ and acts transitively on every metric sphere around $v$ of radius $<1$. We have $\mc{C}(g,\wh T)=T$, which is not a gate-convex subset of $\wh T$. 

By contrast, note that $\overline{\mc{C}}(g,\wh T)=\{v\}$. 
\end{ex}

\begin{ex}
Consider the median space $X=[0,1]^2\setminus\{(0,0)\}$ (endowed with the restriction of the $\ell^1$ metric on $\R^2$). Let $g\in\isom X$ be the reflection in the diagonal through $(0,0)$ and $(1,1)$. Then $\overline{\mc{C}}(g)=[0,1]^2\setminus\left(\{0\}\x[0,1]\cup[0,1]\x\{0\}\right)$, which is not gate-convex in $X$.
\end{ex}

We conclude this subsection with the following result. Although this will not be used any further in this paper, we think it is likely to prove useful in the future (especially since its proof seems to be surprisingly nontrivial).

\begin{prop}\label{cores of subalgebras}
Let $N\cu M$ be a $G$--invariant median subalgebra.
\begin{enumerate}
\item We have $\mc{C}(G,N)\cu N\cap\mc{C}(G,M)$. 
\item If $G\acts N$ has no wall inversions, we have $\overline{\mc{C}}(G,N)\cu N\cap\overline{\mc{C}}(G,M)$.
\item If $G\acts M$ is non-transverse, then $\mc{C}(G,N)=N\cap\mc{C}(G,M)$ and $\overline{\mc{C}}(G,N)\supseteq N\cap\overline{\mc{C}}(G,M)$.
\end{enumerate}
\end{prop}
\begin{proof} 
We can assume that $N$ is nonempty, since the proposition is trivial otherwise. 

We begin by proving parts~(1) and~(2). They are immediate from Claim~1:

\smallskip
{\bf Claim~1:} \emph{If $\mf{k}\in\mc{H}_{1/2}(G,M)$ (resp.\ if $G\acts N$ has no wall inversions and $\mf{k}\in\overline{\mc{H}}_{1/2}(G,M)$), then either $N\cu\mf{k}$ or $\mf{k}\cap N\in\mc{H}_{1/2}(G,N)$ (resp.\ $\mf{k}\cap N\in\overline{\mc{H}}_{1/2}(G,N)$).}

\smallskip \noindent
\emph{Proof of Claim~1.} Recall that $\mc{H}_{1/2}(G,M)\cu\overline{\mc{H}}_{1/2}(G,M)$. If $\mf{k}\in\overline{\mc{H}}_{1/2}(G,M)$, then there exists $g\in G$ such that $g\mf{k}^*\cap\mf{k}^*=\emptyset$, so $\mf{k}^*$ cannot contain the nonempty $G$--invariant set $N$. Thus, either $N\cu\mf{k}$, or $\mf{h}:=\mf{k}\cap N$ is a halfspace of $N$. By Remark~\ref{dynamics in subalgebras}, $\mf{h}$ does not lie in $\mc{H}_1(G,N)$. If $G\acts N$ has no wall inversions, the fact that $g\mf{h}^*\cap\mf{h}^*=\emptyset$ implies that $\mf{h}\not\in\overline{\mc{H}}_0(G,N)$. Hence $\mf{h}\in\overline{\mc{H}}_{1/2}(G,N)$.

If $\mf{k}\in\mc{H}_{1/2}(G,M)$, the halfspace $\mf{h}$ cannot lie in $\mc{H}_0(G,N)$. Otherwise, a finite-index subgroup $H\leq G$ would preserve $\mf{h}$. By Remark~\ref{G and h rmk}, we would have $\mc{H}_{1/2}(G,M)=\mc{H}_{1/2}(H,M)$, hence there would exist $h\in H$ with $h\mf{k}^*\cap\mf{k}^*=\emptyset$, contradicting that $h\mf{h}=\mf{h}$. Thus, $\mf{h}\in\mc{H}_{1/2}(G,N)$. 
\hfill$\blacksquare$

\smallskip
In order to conclude the proof of the proposition, we need to obtain the reverse inclusions $\mc{C}(G,N)\supseteq N\cap\mc{C}(G,M)$ and $\overline{\mc{C}}(G,N)\supseteq N\cap\overline{\mc{C}}(G,M)$ when $G\acts M$ is non-transverse. These will follow once we show that, for every halfspace $\mf{h}\in\mc{H}_{1/2}(G,N)$ (resp.\ $\mf{h}\in\overline{\mc{H}}_{1/2}(G,N)$), there exists $\mf{k}\in\mc{H}_{1/2}(G,M)$ (resp.\ $\mf{k}\in\overline{\mc{H}}_{1/2}(G,M)$) with $\mf{h}=\mf{k}\cap N$. First, we need two more claims.

\smallskip
{\bf Claim~2:} \emph{Consider $\mf{k}\in\mscr{H}_N(M)$ with $\mf{k}\cap N\in\overline{\mc{H}}_{1/2}(G,N)$. If $\mf{k}_1,\mf{k}_2,\mf{k}_3\in G\cdot\mf{k}\cup G\cdot\mf{k}^*$ are pairwise disjoint, then $\mf{k}_1,\mf{k}_2,\mf{k}_3\in G\cdot\mf{k}^*$.}

\smallskip \noindent
\emph{Proof of Claim~2.} 
If the claim fails, there exist $\mf{k}_1\in G\cdot\mf{k}^*$ and disjoint halfspaces $\mf{k}_2,\mf{k}_3\in G\cdot\mf{k}\cup G\cdot\mf{k}^*$ contained in $\mf{k}_1$. Since $\mf{k}\cap N\in\overline{\mc{H}}_{1/2}(G,N)$, no two elements of $G\cdot\mf{k}$ are disjoint (by Remark~\ref{1/2 bar rmk}). Thus, $\mf{k}_2$ and $\mf{k}_3$ must both lie in $G\cdot\mf{k}^*$. Since $\mf{k}\in\mscr{H}_N(M)$, there exist points $x,y,z\in N$ lying, respectively, in $\mf{k}_1^*$, $\mf{k}_2$, $\mf{k}_3$. The median $m(x,y,z)\in N$ lies in $\mf{k}_1\cap\mf{k}_2^*$. Hence $\mf{k}_2\cap N\subsetneq\mf{k}_1\cap N$, which contradicts the fact that $\mf{k}\cap N\not\in\mc{H}_1(G,N)$. 
\hfill$\blacksquare$

\smallskip
{\bf Claim~3:} \emph{Consider $\mf{k}\in\mscr{H}_N(M)$ with $\mf{k}\cap N\in\overline{\mc{H}}_{1/2}(G,N)$. Then the orbit $G\cdot\mf{k}$ can be partitioned into chains $\mscr{C}_i$ such that any two halfspaces in different chains are facing.}

\smallskip \noindent
\emph{Proof of Claim~3.} By Zorn's lemma, there exists a maximal subset $\mc{F}\cu G\cdot\mf{k}$ of pairwise-facing halfspaces. Since $\mf{k}\cap N\in\overline{\mc{H}}_{1/2}(G,N)$, there exists $g\in G$ such that $g\mf{k}\cap N$ and $\mf{k}\cap N$ are facing. Since $G\acts M$ is non-transverse, no two halfspaces in $G\cdot\mf{k}$ are transverse, and we conclude that $g\mf{k}$ and $\mf{k}$ are facing. This shows that $\#\mc{F}\geq 2$.

Since $\mf{k}\cap N\in\overline{\mc{H}}_{1/2}(G,N)$, Remark~\ref{1/2 bar rmk} shows that $(\mf{k}\cap N)^*$ does not lie in the $G$--orbit of $\mf{k}\cap N$. In particular, $\mf{k}^*\not\in G\cdot\mf{k}$. By maximality of $\mc{F}$, for every $\mf{j}\in G\cdot\mf{k}\setminus\mc{F}$, there exists $\mf{f}\in\mc{F}$ such that $\mf{j}$ and $\mf{f}$ are not facing. Since $\mf{f}\neq\mf{j}^*$, this is equivalent to saying that $\mf{j}$ does not contain $\mf{f}^*$. Let us show that the halfspace $\mf{f}$ is uniquely determined by $\mf{j}$. 

Suppose for the sake of contradiction that there exist $\mf{f}_1,\mf{f}_2\in\mc{F}$ and a halfspace $\mf{j}\in G\cdot\mf{k}\setminus\mc{F}$ that contains neither $\mf{f}_1^*$ nor $\mf{f}_2^*$. Since $\mf{f}_1$ and $\mf{f}_2$ are facing, $\mf{f}_1^*$ and $\mf{f}_2^*$ are disjoint.  By Claim~2, $\mf{j}$ cannot be disjoint from both $\mf{f}_1^*$ and $\mf{f}_2^*$. Since $G\acts M$ is non-transverse, $\mf{j}$ is transverse to neither $\mf{f}_1^*$ nor $\mf{f}_2^*$, hence $\mf{j}$ must be contained in either $\mf{f}_1^*$ or $\mf{f}_2^*$. However, since $(\mf{k}\cap N)^*$ does not lie in the $G$--orbit of $\mf{k}\cap N$, we have $\mf{f}_1\cap\mf{f}_2\cap N\neq\emptyset$. So this implies that either $\mf{j}\cap N\subsetneq\mf{f}_1\cap N$ or $\mf{j}\cap N\subsetneq\mf{f}_2\cap N$, contradicting the fact that $\mf{k}\cap N\not\in\mc{H}_1(G,N)$.

Now, for each $\mf{f}\in\mc{F}$, let $\mscr{C}(\mf{f})\cu G\cdot\mf{k}$ be the subset of halfspaces that do not face $\mf{f}$ (including $\mf{f}$ itself). The above discussion shows that the sets $\mscr{C}(\mf{f})$ are pairwise disjoint and that their union is the entire $G\cdot\mf{k}$. Let us show that $\mscr{C}(\mf{f})$ is a chain.

Every halfspace in $\mscr{C}(\mf{f})$ faces every halfspace in $\mc{F}\setminus\{\mf{f}\}$. Since $\#\mc{F}\geq 2$, this implies that the halfspaces in $\mscr{C}(\mf{f})$ pairwise intersect. 
No two elements of $\mscr{C}(\mf{f})$ are transverse. Thus, if $\mscr{C}(\mf{f})$ were not a chain, it would contain halfspaces $\mf{k}_1,\mf{k}_2$ with $\mf{k}_1^*\cap\mf{k}_2^*=\emptyset$. Since $\mf{f}$ does not face any element of $\mscr{C}(f)$, the halfspaces $\mf{k}_1$ and $\mf{k}_2$ would both contain $\mf{f}$, contradicting Claim~2.
 
Finally, let us show that, if $\mf{f}_1,\mf{f}_2\in\mc{F}$ are distinct, then each $\mf{k}_1\in\mscr{C}(\mf{f}_1)$ faces each $\mf{k}_2\in\mscr{C}(\mf{f}_2)$. Since $\mf{f}_2^*$ is contained in $\mf{k}_1$, but not in $\mf{k}_2$, we cannot have $\mf{k}_1\cu\mf{k}_2$. Similarly, we do not have $\mf{k}_2\cu\mf{k}_1$. Thus, if $\mf{k}_1$ and $\mf{k}_2$ were not facing, they would have to be disjoint. Then $\mf{k}_2\cu\mf{k}_1^*\cu\mf{f}_2$ and, since $(\mf{k}\cap N)^*$ does not lie in the $G$--orbit of $\mf{k}\cap N$, we have $\mf{k}_2\cap N\subsetneq\mf{k}_1^*\cap N\subsetneq\mf{f}_2\cap N$. Again, this contradicts the fact that $\mf{k}\cap N\not\in\mc{H}_1(G,N)$. 
\hfill$\blacksquare$

\smallskip
Now, consider a halfspace $\mf{h}\in\overline{\mc{H}}_{1/2}(G,N)$. By Remark~\ref{halfspaces of subsets}, there exists $\mf{k}\in\mscr{H}_N(M)$ with $\mf{h}=\mf{k}\cap N$. Note that the chains $\mscr{C}_i$ provided by Claim~3 are the only maximal totally-ordered subsets of $G\cdot\mf{k}$. Thus, the partition provided by the claim is unique and $G$ permutes the chains $\mscr{C}_i$. 

Let $\mf{g}_i$ be the union of all halfspaces in $\mscr{C}_i$. Since $\mf{h}=\mf{k}\cap N\not\in\mc{H}_1(G,N)$, all halfspaces in $\mscr{C}_i$ have the same intersection with $N$. Thus, $\mf{g}_i\cap N=\mf{k}_i\cap N$ for some $\mf{k}_i\in G\cdot\mf{k}$. In particular, $\mf{g}_i$ is not the entire $M$, hence $\mf{g}_i\in\mscr{H}(M)$. Since $G$ permutes the halfspaces $\mf{g}_i$, which are pairwise facing, we have $\mf{g}_i\in\overline{\mc{H}}_{1/2}(G,M)$. We conclude that there exists $\mf{g}\in\overline{\mc{H}}_{1/2}(G,M)$ with $\mf{g}\cap N=\mf{k}\cap N=\mf{h}$.

Finally, if $\mf{h}\in\mc{H}_{1/2}(G,N)$, we must have $\mf{g}\in\mc{H}_{1/2}(G,M)$. Otherwise, $\mf{g}$ would lie in $\mf{g}\in\mc{H}_0(G,M)$. This would clearly imply that $\mf{h}=\mf{g}\cap N$ lies in $\mc{H}_0(G,N)$, a contradiction.
\end{proof}

\begin{ex}
The following show that the various hypotheses in Proposition~\ref{cores of subalgebras} are necessary.
\begin{enumerate}
\item If $G\acts N$ has wall inversions, we can have $\overline{\mc{C}}(G,N)\not\cu N\cap\overline{\mc{C}}(G,M)$. Simply consider the situation in Example~\ref{empty reduced core} with $M=[-1,1]$ and $N=[-1,1]\setminus\{0\}$. Then $\overline{\mc{C}}(G,M)=\{0\}$ and $\overline{\mc{C}}(G,N)=\emptyset$.
\item If $G\acts M$ is not non-transverse, we can have $\overline{\mc{C}}(G,N)\not\supseteq N\cap\overline{\mc{C}}(G,M)$. Take $M=[0,1]^2$, $N=\{(0,0),(0,1),(1,0)\}$, and let $G$ be the group generated by the reflection in the diagonal through $(0,0)$ and $(1,1)$. Then $\overline{\mc{C}}(G,M)=M$ and $\overline{\mc{C}}(G,N)=\{(0,0)\}$.
\item If $G\acts M$ is not non-transverse, we can also have $\mc{C}(G,N)\not\supseteq N\cap\mc{C}(G,M)$. Let $\varphi$ be a homeomorphism of $[-1,1]$ that fixes $-1$, $0$, $1$, and no other point. Consider $M=[-1,1]^2$ and $N=[-1,1]^2\setminus[-1,0)^2$. Let $G\leq\aut M$ be the group generated by the transformation $(x,y)\mapsto (\varphi(y),\varphi(x))$. Then $\mc{C}(G,M)=M$ and $\mc{C}(G,N)=[0,1]^2$.
\end{enumerate}
\end{ex}

\begin{rmk}
In part~(3) of Proposition~\ref{cores of subalgebras}, the equality $N\cap\mc{C}(G,M)=\mc{C}(G,N)$ also holds if, instead of requiring the action $G\acts M$ to be non-transverse, we ask that $M$ be the underlying median algebra of a median space $X$ equipped with an isometric $G$--action.

Indeed, it then follows from part~(1) of Lemma~\ref{H_0 perp H_1} below that, if $\mf{k}\in\mc{H}_1(G,M)$, then $\mf{k}\cap N$ lies in $\mc{H}_1(G,N)$. It is thus immediate that every element of $\mc{H}_{1/2}(G,N)$ is the intersection with $N$ of an element of $\mc{H}_{1/2}(G,M)$ (which is what occupied most of the proof of Proposition~\ref{cores of subalgebras}).
\end{rmk} 

The following is a consequence of Proposition~\ref{abstract core} and the first two parts of Proposition~\ref{cores of subalgebras}.

\begin{cor}\label{cores intersect subalgebras}
If $G$ is finitely generated, then $\mc{C}(G)$ intersects every nonempty, $G$--invariant subalgebra of $M$. Moreover, $\overline{\mc{C}}(G)$ intersects every nonempty, $G$--invariant subalgebra $N\cu M$ such that $G\acts N$ has no wall inversions.
\end{cor}

\section{Compatible metrics.}\label{compatible metrics sect}

Let $M$ be a finite-rank median algebra. Let $G\acts M$ be an action by median automorphisms. In this section, we consider $G$--invariant pseudo-metrics on $M$ that are compatible with the median operator. The main results are Corollary~\ref{core splitting} and Proposition~\ref{core of g when metric}. We also prove Corollary~\ref{main median semisimple}.

\begin{defn}\label{compatible pseudo-metric defn}
A pseudo-metric $\eta\colon M\x M\ra[0,+\infty)$ is \emph{compatible} if, for all $x,y,z\in M$, we have $\eta(x,y)=\eta(x,m(x,y,z))+\eta(m(x,y,z),y)$. 
\end{defn}

Let $\mc{PD}(M)$ and $\mc{D}(M)$ be the sets, respectively, of all compatible pseudo-metrics on $M$ and of all compatible metrics. To avoid confusion, we will denote metrics by the letter $\delta$ and pseudo-metrics by $\eta$. We write $\mc{PD}(M)^G$ and $\mc{D}(M)^G$ for the sets of \emph{$G$--invariant} pseudo-metrics and metrics. 

If $\delta\in\mc{D}(M)$, then the pair $(M,\delta)$ is a median space (not necessarily a complete or connected one). In fact, one can view every median space as the data of an underlying median algebra and a compatible metric on it. We have $\mc{D}(M)^G\neq\emptyset$ exactly when the action $G\acts M$ arises from an isometric $G$--action on a median space. 

\begin{rmk}\label{metric completion rmk}
Consider $\delta\in\mc{D}(M)$ and let $(X,\delta)$ be the metric completion of $(M,\delta)$. Then $X$ is a complete median space with $\rk X=\rk M$. This follows from \cite[Proposition~2.21]{CDH} and \cite[Lemma~2.5]{Fio1}. Note that the subset $M\cu X$ is a dense median subalgebra.
\end{rmk}

\begin{rmk}\label{connected inversions}
If there exists $\delta\in\mc{D}(M)^G$ such that $(M,\delta)$ is connected, then $G\acts M$ has no wall inversions. Indeed, by Remark~\ref{metric completion rmk}, the metric completion $X$ of $(M,\delta)$ is a complete, finite-rank median space. By \cite[Proposition~B]{Fio1}, every halfspace of $X$ is either open or closed. By part~(1) of Remark~\ref{halfspaces of subsets}, every halfspace of $M$ is then either open or closed in $M$. Since $M$ is connected, a proper subset of $M$ and its complement can never be both closed or both open. Hence $G\acts M$ has no wall inversions.
\end{rmk}

\begin{rmk}\label{metric quotient rmk}
Consider $\eta\in\mc{PD}(M)$ and let $q\colon (M,\eta)\ra (X,\overline\eta)$ be the quotient metric space. Then $(X,\overline\eta)$ is a median space and $q\colon M\ra X$ is a surjective median morphism. By Remark~\ref{median homo rmk}, we have $\rk X\leq\rk M$.
\end{rmk}

\begin{lem}\label{H_0 perp H_1}
If $\mc{D}(M)^G\neq\emptyset$, then: 
\begin{enumerate}
\item for every $g\in G$ and $\mf{h}\in\mc{H}_1(g)$, we have $\bigcap_{n\in\Z}g^n\mf{h}=\emptyset$;
\item the sets $\mc{H}_0(G)$ and $\mc{H}_1(G)$ are transverse.
\end{enumerate}
\end{lem}
\begin{proof}
Part~(2) immediately follows from part~(1) and Remark~\ref{empty intersection implies transversality}, so we only prove part~(1).

Consider $g\in G$ and $\mf{h}\in\mc{H}_1(g,M)$. Pick $\delta\in\mc{D}(M)^G$ and let $X$ be the metric completion of $(M,\delta)$. By Remark~\ref{metric completion rmk}, this is a complete, finite-rank median space with an isometric $G$--action. By Remark~\ref{dynamics in subalgebras}, there exists $\mf{k}\in\mc{H}_1(g,X)$ such that $\mf{h}=\mf{k}\cap M$. Let $n\in\Z$ be such that $g^n\mf{k}\subsetneq\mf{k}$. \cite[Proposition~B]{Fio1} ensures that $\delta(g^{nk}\mf{k},\mf{k}^*)>0$ for some $k\geq 1$. Hence $\delta(g^{nk}\mf{h},\mf{h}^*)\geq\delta(g^{nk}\mf{k},\mf{k}^*)$ diverges for $k\ra +\infty$, proving part~(1).
\end{proof}

Lemmas~\ref{H_0 perp H_1} and~\ref{product median algebras} immediately imply:

\begin{cor}\label{core splitting}
If $\mc{D}(M)^G\neq\emptyset$, then the partitions in Lem\-ma~\ref{halfspaces of core} give rise to $G$--invariant splittings $\mc{C}(G)=\mc{C}_0(G)\x\mc{C}_1(G)$ and $\overline{\mc{C}}(G)=\overline{\mc{C}}_0(G)\x\mc{C}_1(G)$.
\end{cor}

\begin{lem}\label{core of g when metric lem}
Consider $g\in\aut M$ with $\mc{D}(M)^{\langle g\rangle}\neq\emptyset$. Let $\mc{W}_1(g)$ be the set of walls determined by $\mc{H}_1(g)$. Then, for all $x\in\Min(g)$ and $y\in M$, we have:
\[\bigsqcup_{n\in\Z}\mscr{W}(g^nx|g^{n+1}x)=\mc{W}_1(g)\cu\bigcup_{n\in\Z}\mscr{W}(g^ny|g^{n+1}y).\]
\end{lem}
\begin{proof}
Part~(1) of Lemma~\ref{H_0 perp H_1} shows that every $\mf{h}\in\mc{H}_1(g)$ satisfies $\bigcap_{n\in\Z}g^n\mf{h}=\emptyset$. Thus, for every $y\in M$, the orbit $\langle g\rangle\cdot y$ intersects both $\mf{h}$ and $\mf{h}^*$. If $\mf{w}$ is the wall associated to $\mf{h}$, it follows that there exists $n\in\Z$ such that $\mf{w}\in\mscr{W}(g^ny|g^{n+1}y)$. This shows that $\mc{W}_1(g)\cu\bigcup_{n\in\Z}\mscr{W}(g^ny|g^{n+1}y)$ for every $y\in M$.

We are left to show that, for every $x\in\Min(g)$, we have $\mscr{W}(x|gx)\cu\mc{W}_1(g)$. Observe that, if $\mf{k}\in\mscr{H}(x|gx)$, then we have $g^nx\in\mf{k}$ for all $n\geq 1$, and $g^nx\in\mf{k}^*$ for all $n\leq 0$. Hence, for every $k\in\Z$, the two intersections $g^k\mf{k}\cap\mf{k}$ and $g^k\mf{k}^*\cap\mf{k}^*$ are both nonempty. In particular, $\mf{k}\not\in\overline{\mc{H}}_{1/2}(g)\sqcup\overline{\mc{H}}_{1/2}(g)^*$. We also have $\mf{k}\not\in\overline{\mc{H}}_0(g)$, since a proper power of $g$ cannot stabilise $\mf{k}$. We conclude that $\mscr{H}(x|gx)\cu\mc{H}_1(g)$, which completes the proof.
\end{proof}

\begin{ex}
Lemma~\ref{core of g when metric lem} can fail if $\mc{D}(M)^{\langle g\rangle}=\emptyset$. It suffices to consider $M=\R$ and let $g$ be an orientation-preserving homeomorphism with a single fixed point.
\end{ex}

Given $g\in\aut M$ and $\eta\in\mc{PD}^{\langle g\rangle}(M)$, we define the \emph{translation length}:
\[\ell(g,\eta):=\inf_{x\in M}\eta(x,gx)\in[0,+\infty).\]
The following is the main result of this section.

\begin{prop}\label{core of g when metric}
Let $g\in\aut M$ act stably without inversions.
\begin{enumerate}
\item If $\eta\in\mc{PD}(M)^{\langle g\rangle}$, then $\eta(x,gx)=\ell(g,\eta)$ for all $x\in\Min(g)$. Moreover, $\ell(g^n,\eta)=|n|\cdot\ell(g,\eta)$.
\item If $\delta\in\mc{D}(M)^{\langle g\rangle}$ and $y\in M$, then $y\in\Min(g)$ if and only if $\delta(y,gy)=\ell(g,\delta)$. 
\item If $g$ acts non-transversely, then, for all $\eta\in\mc{PD}(M)^{\langle g\rangle}$ and all $y\in M$, we have:
\[\eta(y,gy)=\ell(g,\eta)+2\eta(y,\overline{\mc{C}}(g)).\]
\end{enumerate} 
\end{prop}
\begin{proof}
We begin with part~(1). We can assume that $\eta$ is a genuine metric. Indeed, if $q\colon M\ra X$ is the quotient median space, it follows from Remarks~\ref{metric quotient rmk} and~\ref{median homo rmk} that $q(\Min(g,M))\cu\Min(g,X)$.

Thus, it suffices to consider $\delta\in\mc{D}(M)^{\langle g\rangle}$. By \cite[Theorem~5.1]{CDH}, there exists a $\langle g\rangle$--invariant measure $\mu$ on $\mscr{W}(M)$ such that $\delta(u,w)=\mu(\mscr{W}(u|w))$ for all $u,w\in M$. Recall that, by Corollary~\ref{semisimple cor general}, the set $\Min(g)$ is nonempty. Consider points $x\in\Min(g)$ and $y\in M$.

Let $\Om=\bigcup_{n\in\Z}\mscr{W}(g^ny|g^{n+1}y)\cu\mscr{W}(M)$. By Lemma~\ref{core of g when metric lem}, $\Om$ contains $\mscr{W}(y|gy)$ and $\mscr{W}(x|gx)$. Moreover, every $\langle g\rangle$--orbit in $\Om$ intersects $\mscr{W}(y|gy)$ in at least one element, while it intersects $\mscr{W}(x|gx)$ in at most one element (since the sets $\mscr{W}(g^nx|g^{n+1}x)$ are pairwise disjoint, as $x\in\Min(g)$). Recalling that $\mu$ is $\langle g\rangle$--invariant, it follows that:
\[\delta(y,gy)=\mu(\mscr{W}(y|gy))\geq\mu(\mscr{W}(x|gx))=\delta(x,gx).\]
This inequality holds for every $x\in\Min(g)$ and $y\in M$, so $\delta(x,gx)=\ell(g,\delta)$ for all $x\in\Min(g)$. 

For every $n\in\Z\setminus\{0\}$, we have $\Min(g)\cu\Min(g^n)$. Thus, if $\eta\in\mc{PD}(M)^{\langle g\rangle}$ and $x\in\Min(g)$, then:
\[\ell(g^n,\eta)=\eta(x,g^nx)=|n|\cdot\eta(x,gx)=|n|\cdot\ell(g,\eta).\]
This proves part~(1).

We now address part~(2). Consider $\delta\in\mc{D}(M)^{\langle g\rangle}$ and $y\in M$ with $\delta(y,gy)=\ell(g,\delta)$. The triangle inequality gives $\delta(y,g^ny)\leq n\cdot\ell(g,\delta)$ for all $n\geq 1$. Part~(1) implies that $\delta(y,g^ny)=n\cdot\ell(g,\delta)$. Thus, for all $k\leq m\leq n$ in $\Z$, we have $\delta(g^ky,g^ny)=\delta(g^ky,g^my)+\delta(g^my,g^ny)$. This implies that $m(g^ky,g^my,g^ny)=g^my$. Hence the sets $\mscr{W}(g^iy|g^{i+1}y)$ are pairwise disjoint, and $y$ must lie in $\Min(g)$. This proves part~(2).

Finally, let us prove part~(3). By Proposition~\ref{gate-convex C bar}, the set $\overline{\mc{C}}(g)$ is gate-convex. Let $\pi\colon M\ra\overline{\mc{C}}(g)$ be the gate-projection. Since $\overline{\mc{C}}(g)$ is $\langle g\rangle$--invariant, we have $\pi(gy)=g\pi(y)$ for every $y\in M$. Hence:
\[\mscr{W}(y|gy)\cu\mscr{W}(y|\pi(y))\cup\mscr{W}(\pi(y)|g\pi(y))\cup\mscr{W}(\pi(gy)|gy).\]
Since $\pi$ is the gate-projection, we have:
\begin{align*}
\mscr{W}(y|g\pi(y))&=\mscr{W}(y|\pi(y))\sqcup\mscr{W}(\pi(y)|g\pi(y)), & \mscr{W}(\pi(y)|gy)&=\mscr{W}(\pi(y)|g\pi(y))\sqcup\mscr{W}(\pi(gy)|gy).
\end{align*}
By Proposition~\ref{non-transverse prop}, the set $\mscr{W}(y|\pi(y))\cap\mscr{W}(gy|\pi(gy))=\mscr{W}(y,gy|\overline{\mc{C}}(g))$ is empty, hence:
\[\mscr{W}(y|gy)=\mscr{W}(y|\pi(y))\sqcup\mscr{W}(\pi(y)|g\pi(y))\sqcup\mscr{W}(\pi(gy)|gy).\]
Part~(1) and Proposition~\ref{non-transverse prop} yield $\eta(x,gx)=\ell(g,\eta)$ for all $\eta\in\mc{PD}(M)^{\langle g\rangle}$ and $x\in\Min(g)$. Thus:
\[\eta(y,gy)=\eta(y,\pi(y))+\eta(\pi(y),g\pi(y))+\eta(\pi(gy),gy)=\ell(g,\eta)+2\eta(y,\overline{\mc{C}}(g)).\]
\end{proof}

We can now deduce Corollary~\ref{main median semisimple} from Corollary~\ref{semisimple cor general}.

\begin{proof}[Proof of Corollary~\ref{main median semisimple}]
Since $X$ is connected, Remark~\ref{connected inversions} shows that $g\in\isom X$ acts stably without inversions. By Corollary~\ref{semisimple cor general}, there exists a point $x\in\Min(g)$. By part~(1) of Proposition~\ref{core of g when metric}, the point $x$ realises the translation length of $g$.

If $X$ is geodesic and $g$ does not fix $x$, let $\alpha$ be a geodesic segment joining $x$ and $gx$. Since $x\in\Min(g)$, the segments $g^n\alpha$ glue to form the required axis of $g$.
\end{proof}

The following motivates our notation from Subsection~\ref{dynamics sect} by comparing it to \cite{CS}.

\begin{rmk}\label{explanation of 0,1,1/2}
Suppose that there exists $\delta\in\mc{D}(M)^G$. Fix a basepoint $x_0\in M$. By analogy with \cite[Subsection~3.3]{CS}, one could call a halfspace $\mf{h}\in\mscr{H}(M)$:
\begin{itemize}
\item \emph{fully inessential} if the points of $G\cdot x_0$ have uniformly bounded distance from $\mf{h}$ and $\mf{h}^*$;
\item \emph{fully essential} if $G\cdot x_0$ contains points arbitrarily far from $\mf{h}^*$ and arbitrarily far from $\mf{h}$;
\item \emph{half-essential} if $G\cdot x_0$ contains points arbitrarily far from $\mf{h}^*$, but $G\cdot x_0$ has uniformly bounded distance from $\mf{h}$.
\end{itemize}
These three classes of halfspaces are closely related to our sets $\mc{H}_0(G)$, $\mc{H}_1(G)$ and $\mc{H}_{1/2}(G)$, respectively, and they motivate our notation. The sets $\mc{H}_{\bullet}(G)$ have the advantage of being evidently independent of the choice of a metric $\delta$.

It is clear that the elements of $\mc{H}_0(G)$ are fully inessential. As in the proof of part~(1) of Lemma~\ref{H_0 perp H_1}, one sees that the elements of $\mc{H}_1(G)$ are fully essential. If $G$ is finitely generated, it follows from part~(1) of Theorem~\ref{abstract core} that no fully essential halfspace can lie in $\mc{H}_{1/2}(G)$. 

In conclusion, if $G$ is finitely generated, then $\mc{H}_1(G)$ always coincides with the set of fully essential halfspaces and $\mc{H}_0(G)$ is always contained in the set of fully inessential halfspaces. However, $\mc{H}_{1/2}(G)$ can contain both half-essential and fully inessential halfspaces (the latter e.g.\ when $M$ is a locally infinite tree $T$ and $G$ fixes an infinite-valence vertex, while acting freely on its neighbours).
\end{rmk}

\section{Actions of polycyclic groups.}\label{polycyclic sect}

This section is devoted to the proof of Corollary~\ref{corintro polycyclic} and its immediate consequence Corollary~\ref{flattorus cor}.

\subsection{Lineal median algebras.}

Let $M$ be a finite-rank median algebra. Recall from Subsection~\ref{ss sect}, that we denote by $\overline M$ the zero-completion of $M$. 

\begin{defn}
Given $\xi^{\pm}\in\overline M$, we say that $M$ is \emph{lineal with endpoints $\xi^{\pm}$} if $M\cu I(\xi^-,\xi^+)$. 
\end{defn}

\begin{lem}\label{lineal lem}
$M$ is lineal if and only if $\mscr{H}(M)$ can be partitioned into two ultrafilters.
\end{lem}
\begin{proof}
If $M$ is lineal with endpoints $\xi^{\pm}$, then the intersections between $\mscr{H}(M)\simeq\mscr{H}_M(\overline M)$ and the subsets $\mscr{H}(\xi^-|\xi^+),\mscr{H}(\xi^+|\xi^-)\cu\mscr{H}(\overline M)$ provide the required partition into two ultrafilters.

Conversely, suppose that $\mscr{H}(M)=\s\sqcup\s^*$ is a partition into two ultrafilters. For every interval $I=I(x,y)\cu M$, we have a partition $\mscr{H}_I(M)=(\s\cap\mscr{H}_I(M))\sqcup(\s^*\cap\mscr{H}_I(M))$. By Remark~\ref{halfspaces of subsets}, this corresponds to a partition $\mscr{H}(I)=\s_I\sqcup\s_I^*$, where $\s_I$ and $\s_I^*$ are ultrafilters.

The sets $\s_I$ and $\s_I^*$ satisfy the hypotheses of Lemma~\ref{nonempty intersection criterion}. In order to see this, consider a chain $\mscr{C}\cu\s_I$. Then either $\mscr{C}\cap\mscr{H}(x|y)$ or $\mscr{C}\cap\mscr{H}(y|x)$ is cofinal in $\mscr{C}$. Hence $\bigcap\mscr{C}$ is nonempty and it is a halfspace of $I$. Note that, given any $\mf{h}\in\mscr{C}$, the halfspace $\mf{h}^*$ lies in $\s_I^*$ and is disjoint from $\bigcap\mscr{C}$. Hence $\bigcap\mscr{C}\not\in\s_I^*$, and $\bigcap\mscr{C}\in\s_I$ as required.

This shows that, for every interval $I\cu M$, there exist two points $z_I^{\pm}\in I$ such that $I=I(z_I^-,z_I^+)$ and $z_I^+$ (resp.\ $z_I^-$) lies in the intersection of all halfspaces in $\s\cap\mscr{H}_I(M)$ (resp.\ in $\s^*\cap\mscr{H}_I(M)$). It is straightforward to check that the functions $I\mapsto z_I^+$ and $I\mapsto z_I^-$ define the required points $\xi^{\pm}$ in the zero-completion of $M$.
\end{proof}

\begin{rmk}\label{finitely many partitions}
If $\mscr{H}(M)=\s\sqcup\s^*=\tau\sqcup\tau^*$ are distinct partitions into ultrafilters, we can write:
\[\mscr{H}(M)=[(\s\cap\tau)\sqcup(\s\cap\tau)^*] \sqcup [(\s^*\cap\tau)\sqcup(\s^*\cap\tau)^*].\]
Note that every halfspace in $\s\cap\tau$ is transverse to every halfspace in $\s^*\cap\tau$. Thus, Lemma~\ref{product median algebras} gives a nontrivial product splitting of $M$. It follows that there are only finitely many distinct partitions of $\mscr{H}(M)$ into two ultrafilters.
\end{rmk}

\begin{rmk}
Suppose $M$ is lineal and $G\acts M$ is an action by median automorphisms. Then there exist a finite-index subgroup $H\leq G$ and $H$--fixed points $\xi^{\pm}\in\overline M$ such that $M\cu I(\xi^-,\xi^+)$. This follows from Remark~\ref{finitely many partitions} and Lemma~\ref{lineal lem}.
\end{rmk}

\subsection{Lineal median spaces.}

We now restrict to the setting of (not necessarily complete or connected) finite-rank median spaces.

The proof of the next lemma requires a certain familiarity with our previous work in \cite{Fio1}. Still, this is a rather predictable result and we hope that the reader will not find this troubling.

\begin{lem}\label{lineal completion lem}
\begin{enumerate}
\item[]
\item If $X$ is a lineal, finite-rank median space, then its metric completion $\wh X$ is also lineal. 
\item Every lineal median space of rank $r$ can be isometrically embedded in $(\R^r,d_{\ell_1})$.
\end{enumerate}
\end{lem}
\begin{proof}
Part~(2) follows from part~(1) and \cite[Proposition~2.19]{Fio1}. 
Thus, we only prove part~(1).

By \cite[Theorem~5.1]{CDH} and \cite[Section~3]{Fio1}, the set $\mscr{H}(\wh X)$ can be equipped with a measure $\mu$ such that $d(x,y)=\mu(\mscr{H}(x|y))$ for all $x,y\in X$. By Lemma~\ref{lineal lem} and Remark~\ref{halfspaces of subsets}, the subset $\mscr{H}_X(\wh X)$ is partitioned into two ultrafilters. In the terminology of \cite[p.\ 1349]{Fio1}, the elements of $\mscr{H}(\wh X)\setminus\mscr{H}_X(\wh X)$ are \emph{non-thick}, hence $\mscr{H}(\wh X)\setminus\mscr{H}_X(\wh X)$ has measure zero by \cite[Corollary~3.7]{Fio1}. 

The fact that $\wh X$ is lineal can now be deduced by retracing the proof of Lemma~\ref{lineal lem} above and observing that, invoking \cite[Corollary~3.11(1)]{Fio1}, it yields the following stronger result: \emph{The median space $X$ is lineal if and only if there is a partition $\mscr{H}(M)=\s\sqcup\s^*\sqcup\mc{K}$, where $\s$ and $\s^*$ are ultrafilters and $\mu(\mc{K})=0$}.
\end{proof}

\begin{lem}\label{isometries of intervals}
If $X$ is a lineal, finite-rank median space, then $\isom X$ is virtually (locally finite)--by--abelian.
\end{lem}
\begin{proof}
If $\wh X$ is the metric completion of $X$, we have an embedding $\isom X\hookrightarrow\isom\wh X$ and $\wh X$ is again lineal by part~(1) of Lemma~\ref{lineal completion lem}. Thus, it is not restrictive to assume that $X$ is complete.

We begin with the following special case. Set $r=\rk X$.

\smallskip
{\bf Claim:} \emph{if $H\leq\isom X$ fixes points $x\in X$ and $\xi\in\overline X$, then $H$ is locally finite.}

\smallskip \noindent
\emph{Proof of Claim.} Denote by $\mc{H}(t)\cu\mscr{H}(x|\xi)\cap\mscr{H}(X)$ the set of halfspaces at distance exactly $t>0$ from $x$. By \cite[Proposition~B]{Fio1}, any chain in $\mc{H}(t)$ has length $\leq 2r$, while any anti-chain has length $\leq r$ by definition of rank. Dilworth's lemma then yields $\#\mc{H}(t)\leq 2r^2$. 

The action $H\acts\mscr{H}(X)$ preserves the set $\mc{H}(t)$ for each $t>0$. This yields a homomorphism $\iota\colon H\ra\prod_{t>0}{\rm Sym}(\mc{H}(t))$. Each element in $\ker\iota$ preserves every halfspace of $X$ at positive distance from $x$, thus fixing every point of $X$. We conclude that $\iota$ is injective.

The claim now follows from the (standard) observation that direct products of uniformly finite groups are locally finite. 
\hfill$\blacksquare$

\smallskip
By Corollary~\ref{product cor}, there is a maximal product splitting $X=X_1\x\dots\x X_k$. The product of the groups $\isom X_i$ sits inside $\isom X$ as a finite-index subgroup. By Lemma~\ref{lineal lem}, each $X_i$ is again lineal. Thus, it is not restrictive to assume that $X$ is irreducible.

Now, the points $\xi,\eta\in\overline X$ such that $X\cu I(\xi,\eta)$ are unique by Remark~\ref{finitely many partitions}. An index--$2$ subgroup of $\isom X$ fixes $\xi$ and $\eta$. By \cite[Theorem~F]{Fio2}, there exist a finite-index subgroup $G\leq\isom X$ and a homomorphism $\chi\colon G\ra\R^r$ such that every finitely generated subgroup of $\ker\chi$ fixes a point of $X$, as well as $\xi$ and $\eta$. 

If $H\leq\ker\chi$ is finitely generated and $x\in X$ is a point fixed by it, we obtain a restriction homomorphism $\rho\colon H\ra\isom I(x,\xi)\x\isom I(x,\eta)$. By the Claim, the image of $\rho$ is finite. On the other hand, the kernel of $\rho$ preserves every halfspace of $I(x,\xi)\cup I(x,\eta)$, hence every halfspace of $X$. It follows that $\rho$ is injective and $H$ is finite. In conclusion, $\ker\chi$ is locally finite. 
\end{proof}

\subsection{Actions on lineal median spaces.}

\begin{defn}
Let $\mc{G}$ be the class of groups $G$ with the following property. If $X$ is a finite-rank median space and $G\acts X$ is an essential isometric action without wall inversions, then $X$ is lineal.
\end{defn}

\begin{rmk}\label{Z in G}
It follows from part~(3) of Lemma~\ref{ss lemma} that $\Z\in\mc{G}$.
\end{rmk}

\begin{rmk}\label{finite index supergroups}
If $H\leq G$ has finite index and $H\in\mc{G}$, then $G\in\mc{G}$. 
Indeed, if an action $G\acts X$ is essential and without wall inversions, then so is the restriction $H\acts X$. 
\end{rmk}

\begin{prop}\label{H and G/H}
Suppose that there exists a subgroup $H\lhd G$ such that $H$ and $G/H$ lie in $\mc{G}$. Suppose moreover that $H$ is finitely generated. Then $G\in\mc{G}$. 
\end{prop}
\begin{proof}
Let $X$ be a finite-rank median space, and let $G\acts X$ be an essential isometric action without wall inversions. Since $H$ is finitely generated, the reduced core $\overline{\mc{C}}(H)$ is nonempty by Theorem~\ref{abstract core}. Since $H$ is normal, Remark~\ref{core and commensurability} shows that $\overline{\mc{C}}(H)$ is $G$--invariant. Part~(1) of Lemma~\ref{essential vs minimal cor} then guarantees that $\overline{\mc{C}}(H)=M$. By Corollary~\ref{core splitting}, we have a product splitting
\[M=\overline{\mc{C}}_0(H)\x\mc{C}_1(H).\]
Note that Remark~\ref{core and commensurability} also guarantees that the sets $\overline{\mc{H}}_0(H)$ and $\mc{H}_1(H)$ are $G$--invariant. So, the factors in the above splitting of $M$ are preserved by $G$. 

Since $H\in\mc{G}$, the space $\mc{C}_1(H)$ is lineal (possibly a single point). By Lemma~\ref{lineal lem}, it suffices to show that $\overline{\mc{C}}_0(H)$ is lineal. By Remark~\ref{factors through finite group}, a finite-index characteristic subgroup $H_0\lhd H$ fixes $\overline{\mc{C}}_0(H)$ pointwise. Thus, the essential action $G\acts\overline{\mc{C}}_0(H)$ factors through an essential action of $G/H_0$. Since $G/H\in\mc{G}$ has finite index in $G/H_0$, Remark~\ref{finite index supergroups} shows that $G/H_0\in\mc{G}$. Hence $\overline{\mc{C}}_0(H)$ is lineal.
\end{proof}

Recall that a group is \emph{polycyclic} if it is solvable and all its subgroups are finitely generated. Proposition~\ref{H and G/H} and Remarks~\ref{Z in G} and~\ref{finite index supergroups} imply the following.

\begin{cor}\label{polycyclic in G}
All virtually polycyclic groups lie in $\mc{G}$.
\end{cor}

\begin{rmk}\label{lineal actions of polycyclic groups}
Let $X$ be a lineal, finite-rank median space. If $G$ is virtually polycyclic, then every isometric action $G\acts X$ factors through a virtually abelian group. This follows from Lemma~\ref{isometries of intervals} and the fact that polycyclic groups are virtually torsion-free \cite[Theorem~3.21]{Hirsch}.
\end{rmk}

We are now ready to prove Corollary~\ref{corintro polycyclic} from the introduction.

\begin{proof}[Proof of Corollary~\ref{corintro polycyclic}]
By Theorem~\ref{abstract core}, the core $\mc{C}(H)$ is nonempty. By Corollary~\ref{core splitting}, we have a splitting $\mc{C}(H)=\mc{C}_0(H)\x\mc{C}_1(H)$. By Remark~\ref{core and commensurability}, the $G$--action leaves $\mc{C}(H)$ invariant and preserves the product splitting. Since $G\acts X$ is minimal, we have $\mc{C}(H)=X$ and we can set $P=\mc{C}_0(H)$ and $F=\mc{C}_1(H)$. By part~(2) of Proposition~\ref{finite orbit new prop}, the action $H\acts P$ fixes a point. The action $H\acts F$ is essential, so $F$ is lineal by Corollary~\ref{polycyclic in G}. By Remark~\ref{lineal actions of polycyclic groups}, the action $H\acts F$ factors through a virtually abelian group. By part~(2) of Lemma~\ref{lineal completion lem}, the space $F$ isometrically embeds in $\R^r$.
\end{proof}

\begin{proof}[Proof of Corollary~\ref{flattorus cor}]
By Remark~\ref{connected inversions}, the action $A\acts X$ has no wall inversions. By Proposition~\ref{H=H_1} and part~(1) of Lemma~\ref{essential vs minimal cor}, there exists an $A$--invariant, proper, convex subset $C\cu X$ such that the action $A\acts C$ is minimal and without wall inversions.

Corollary~\ref{corintro polycyclic} now yields a splitting $C=F\x P$, where $A\acts P$ fixes a point and $F$ isometrically embeds into $\R^n$. Since $A\acts C$ is minimal, $P$ must be a singleton, concluding the proof.
\end{proof}

\bibliography{../mybib}
\bibliographystyle{alpha}

\end{document}